\documentclass[3p,times]{elsarticle}

\usepackage{ecrc}
\usepackage{amsmath}
\usepackage{graphicx,color}
\usepackage{amscd}
\usepackage{amsfonts}
\usepackage{amssymb}
\usepackage{mathrsfs}
\usepackage{mathtools}
\usepackage{paralist}
\usepackage{hyperref}

\usepackage[utf8]{inputenc}
\usepackage[T1]{fontenc}
\usepackage[french,english]{babel} 

\usepackage{setspace}

\newcommand{\capsize}[1]{\scalebox{1.5}{$#1$}}

\usepackage{graphicx,color,amsmath,amsthm,amsopn,amstext,amscd,amsfonts,amssymb}

\newtheorem{thm}{Theorem} 

\newtheorem{prop}{Proposition}[section]
\newtheorem{lem}[prop]{Lemma}
\newtheorem{defn}{Definition}[section]
\newtheorem{rem}{Remark}[section]

\numberwithin{equation}{section}

\newcommand\R{\mathbb R}

\newcommand\N{\mathbb N}

\newcommand{\be}{\begin{equation} \label}
	\newcommand{\ee}{\end{equation}}
\newcommand{\ds}{\displaystyle}
\newcommand{\ts}{\textstyle}

\allowdisplaybreaks
\def\e{\varepsilon}
\def\e{\varepsilon}

\newcommand{\da}{\partial_a}
\newcommand{\dt}{\partial_t}
\newcommand{\nud}{\frac{1}{2}}
\newcommand{\rp}{{\mathbb{R}_+}}
\newcommand{\ddt}[1]{\frac{d#1}{dt}}
\newcommand{\cE}{{\mathcal E}}
\newcommand{\cS}{{\mathcal S}}
\newcommand{\nrm}[2]{{\lVert#1\rVert}_{#2}}

\newcommand{\D}{{\mathcal D}}

\newcommand{\veps}{u_\e}

\newcommand{\zz}{{z_0}}
\newcommand{\rhostar}{\mu_1}
\newcommand{\rhobar}{\varphi}

\def\e{\varepsilon}

\newcommand{\zeps}{z_\e}
\newcommand{\zp}{{z_p}}
\newcommand{\kernel}{\varrho}

\newcommand{\cL}{{\mathcal L}}
\newcommand{\intrp}{\int_0^\infty}

\DeclareMathOperator*{\argmin}{argmin}

\volume{00}

\firstpage{1}

\journalname{}

\runauth{V. Milisic and P. Souplet}

\jid{procs}

\jnltitlelogo{}

\usepackage{amssymb}

\usepackage[figuresright]{rotating}

\begin{document}

\begin{frontmatter}




\title{From transient elastic linkages to friction: \\a complete study of a  penalized fourth order  equation with delay}


\author[1]{Vuk Mili\v{s}i\'c}
\ead{vuk.milisic@univ-brest.fr}
    
\address[1]{Univ. Brest, CNRS UMR 6205 Laboratoire de Math\'ematiques de Bretagne Atlantique, 
		6, avenue Victor Le Gorgeu, 29200, Brest, France}

 \author[2]{Philippe Souplet} 
 \ead{souplet@math.univ-paris13.fr}
\address[2]{Universit\'e Sorbonne Paris Nord, LAGA, CNRS UMR 7539,
	99, avenue Jean-Baptiste Cl\'ement,
	93430, Villetaneuse, France}

\selectlanguage{english} 
\begin{abstract}
\noindent\textbf{(English)}  In this paper we consider a fourth order nonlinear parabolic  delayed problem
modelling a quasi-instantaneous turn-over of linkages in the context
of cell-motility.
The model depends on a small
parameter $\varepsilon$ which represents a typical time scale of the memory effect.
We first prove global existence and uniqueness of solutions for $\varepsilon$ fixed.
This is achieved by combining suitable fixed-point and energy arguments and by uncovering a nonlocal in time,  conserved integral quantity.
After giving a complete classification of steady states in terms of elliptic functions,
we next show that every solution converges to a steady state as $t\to\infty$. 
When $\varepsilon \to 0$, 
we then establish convergence results on finite time intervals, showing 
that the solution tends in a suitable sense
towards the solution of a parabolic problem without delay.
Moreover, we establish the convergence of energies as $\varepsilon \to 0$, 
which 
enables us to show
that, for $\varepsilon$ small enough, 
the $\varepsilon$-dependent 
problem inherits part of the large time asymptotics of the limiting parabolic problem.
\vspace{1em}

\noindent\textbf{(Français)} 
	Dans cet article, nous considérons un problème parabolique avec retard, non linéaire et du quatrième ordre en espace, modélisant le renouvellement quasi-instantané  des liaisons élastiques dans le contexte de la motilité cellulaire. 
	Le modèle dépend d'un petit paramètre $\e$ qui représente une échelle de temps typique de l'effet de mémoire. 
	Nous commençons par prouver l'existence globale et l'unicité des solutions pour $\e$ fixé. 
	Nous y parvenons en combinant des arguments de point fixe et d'énergie appropriés et en exhibant une nouvelle quantité 
	intégrale conservée mais non locale en temps. 
	Après avoir donné une classification complète des états stationnaires en termes de fonctions elliptiques, 
	nous montrons que chaque solution converge vers un état stationnaire lorsque $t\to \infty$. 
	Lorsque $\e \to 0$, nous établissons alors des résultats de convergence pour des intervalles de temps finis, 
	montrant que la solution tend dans un sens approprié vers la solution d'un problème parabolique sans retard. 
	De plus, nous établissons la convergence des énergies quand $\e$ tend vers 0, 
	ce qui nous permet de montrer que, pour $\e$  suffisamment petit, 
	le problème $\e$-dépendant hérite d'une partie de l'asymptotique en temps grand du problème parabolique limite.

\end{abstract}


\end{frontmatter}
\selectlanguage{english} 
\newpage


\section{Introduction}
This article is a contribution to the mathematical study of adhesion forces in the context of cell motility,
in continuation to the project \cite{MiOel.1,MiOel.2,MiOel.3,MiOel.4}. Cell adhesion and migration play a crucial role in many biological phenomena such as embryonic development, inflammatory responses, wound healing and tumor metastasis. 
The main motivation comes from the seminal papers \cite{OeSchVi,OeSch}, where the authors
built a complex and realistic model of the Lamellipodium. It is a cytoskeletal quasi-two-dimensional actin mesh,
and the whole structure propels the cell on the substrate.
 The model has the potential to include the effects of (de)polymerization, of the mechanical
	effects of cross-linking, bundling, and motor proteins, of cell-substrate adhesion, as well as of
	the leading edge of the membrane. As the (de)polymerisation and the edge of the membrane add even more mathematical complexity 
	we simplify the problem and consider the others mechanical effects except these. 
 In this setting, the authors of \cite{OeSchVi,OeSch} consider an axi-symmetric idealization
of the network. It is represented by two families of  inextensible filaments (clockwise and anti-clockwise)
interacting with each other. Thanks to these hypotheses, 
the problem reduces to  a single equation whose solution is the position $z(s,t)\in\R^2$
of a single filament
evolving in time $t\in \rp$ and with respect to a reference configuration 
$s \in (0,L)$. 
Using gradient flow techniques, the authors obtain a  nonlinear equation:
\begin{equation}\label{eq.starting.pbm}
	\left\{ 
	\begin{aligned}
		&	\underbrace{   z_0''''}_{\text{bending}}-
		\underbrace{( \lambda  z_0')'}_{\text{unextensibility}} + \underbrace{
			\mu^A D_t z_0 }_{\text{adhesion }} 
		+	
		\underbrace{
			\left( 
			\left(
			\arccos
			\left(\left| z_0 \right|'\right)-\varphi_0
			\right)  
			(z_0')^{\bot}
			\right)'
		}_{\text{twisting}}
		+ 
		\underbrace{
			D_t \varphi
			z^\bot_0 
		}_{\text{stretching}}
		=0 \\
		& 	\left| \partial_s z_0\right|= 1,\\ 
	\end{aligned}
	\right.
\end{equation}
where $\varphi = \arg z_0$ and $z^\bot_0=(-z_{0,2},z_{0,1})^T$ and $'$ denotes the space derivative (with respect to $s$). The Lagrange multiplier $\lambda(s,t)$ accounts for the inextensibility constraint and is an unknown of the problem.
They complement this equation 
with initial and  natural boundary conditions that we omit for sake of conciseness.
The adhesion and stretching friction terms appearing in this force balance equation
were obtained as formal limits of a microscopic description of adhesion mechanisms.
For instance the first friction term 
is obtained as the limit when $\e$, a small dimensionless parameter, goes to zero in the expression:
$$
\frac{1}{\e} \intrp \left(z_\e (s,t) - z_\e (s+\e a,t-\e a)\right) \kernel(s,a,t) da \to \mu_1 (s,t) \left( \dt - \partial_s \right) z_0(s,t) =: \mu^A D_t z.
$$
Here $a$ represents the age of a linkage established with the previous locations
of the filament and the shift in the reference configuration $s+ \e a$ comes from 
the (de)polymerization of the filaments inside the lamellipodium. 
The  parameter $\e$ represents at the same time a characteristic age of the linkages, 
and  the inverse of their stiffness.
The elongation $(z(s,t)-z(s+\e a,t-\e a))/\e$, for a fixed $a$, is the linear elastic force 
exerted on $z(s,t)$ by a linkage established at $z(s+\e a,t-\e a)$ (see also 
\cite{oelz_schmei_book,MR3385931} for an extension to the non-axi-symmetric case). 

The type of system can be considered as a model for the motion of elastic, inextensible rods in
a high friction regime \cite{oelz_schmei_book}. Currently the mathematical modelling of biopolymers and biopolymer
networks is a field of high scientific interest and elastic rod models have recently also been used
for the modelling of the DNA \cite{balaeff2006modeling,bijani2006anisotropic}. 
The primary motivation for these studies
is to obtain qualitative results of these models which have the potential to give
insight into the behaviour of the respective biological systems.

More than ten years ago, our initial goal was to 
give a rigorous justification of the limit when $\e$ goes to zero in the previous 
system leading to \eqref{eq.starting.pbm},  but the nonlinear
space-dependent feature in \eqref{eq.starting.pbm} made it out of reach at that time.
A first attempt to cope with the fully nonlinear space-dependent 
problem was performed in \cite{Mi.5}.
In this latter article the first author considered delayed harmonic maps 
($z''''$ is replaced by $z''$ and the constraint $|z'(s,t)|=1$ is replaced by $|z(s,t)|=1$ for a.e.~$s \in (0,L)$)
and showed  rigorously the asymptotics with respect to $\e$. 
In that
work a key feature is compactness in time. When compared with the classical
parabolic problem without delay \cite{oelz.sema}, compactness is far more difficult to obtain.
Indeed, in the standard case, using minimizing movements \cite{AmGiSa}, 
a variational principle provides this estimate first, 
and convergence with respect to the discretisation step of the  semi-discrete scheme in time 
is then immediate. This provides in turn existence (and uniqueness) for 
the continuous problem \eqref{eq.starting.pbm} for instance (this was exactly 
the method used in \cite{OeSch} when starting from classical {\em minimizing movements} {\em \`a la} De Giorgi \cite{AmGiSa}). 

The present work
is the first attempt in order to handle the fourth order case from \cite{OeSch},
 and this turns out to be quite challenging mathematically.
Because the geometric structure (related to the constraint) of the problem 
is different, results from \cite{Mi.5} do not apply.
Here we consider a penalization of the constraint $|z'(s,t)|=1$ for a.e.~$t>0$ and all $s \in I$.
Namely we
replace it by a non-convex potential $F_\delta(z'(s,t)) = \delta^{-1}(1-{z'}^{2}(s,t))^2$ in the variational principle.
When $\delta$ becomes small the solution $z^{\delta}$ should in some sense converge to the
solution of the constrained problem above \cite{oelz.sema}.
These hypotheses lead to study a delayed
parabolic model of Allen-Cahn type (but of fourth order in space) where the penalization term applies to $z'$, that we detail 
in the next subsection.

For the usual Allen-Cahn equation $u_t-\Delta u=\delta^{-1}u(1-u^2)$
(second order in space, local in time and with double-well potential nonlinearity acting on $u$ itself),
the long time and/or singular perturbation behaviors have been studied in, e.g., the classical works 
\cite{AC79, BK91, Ch92, MS95, AlHiMa}. As for the
literature on linear and non-linear, nonlocal in time Volterra-type equations, it is vast in finite dimension
	(cf.~\cite{Gripenberg_ea} and references included). However, when the space variable is added as well as
	for instance elliptic operators \cite{DiVaVe,DuTu,AlHiMa,NguCaTu}) 
	the time operators considered are often either fractional derivatives or delay operators
	independent of these space variables. Exceptions are papers 
	where the fractional order depends on the space variable; see for instance \cite{BoHeZa, ZW21}. 
	 In the problem that we consider here, 
	the density of linkages $\varrho$ depends on the space variable as well, which for instance forbids
	any kind of space-time variable separation. This assumption allows to take into account adhesion forces related for example to the properties of the substrate or specific features of the proteins involved, whereas on the mathematical 
	side it creates some more complications when analyzing the equations. Moreover, the non-linearity 
    that we consider concerns first order derivatives in space which is original and to our knowledge
    poorly understood \cite{Oelz.CMS}.

\subsection{The model}
We define $I := (0,L)$, with $L>0$, to be a segment of the real line.
We now assume that the filament at time $t$ is described by the 
	graph of a planar curve $I\ni s\mapsto z(s,t)\in\R$.
	First consider the time dependent energy (depending on the past positions of the filament),
	defined by
	\be{energy-kappa}
\begin{aligned}
		\tilde\cE_t (w) := \frac{1}{2 \e}& \int_I \intrp  \bigl( w(s) - \zeps(s,t-\e a)\bigr)^2 \kernel(s,a) da ds
 	+ \int_I \kappa^2(w(s)) ds + \frac{1}{\delta}\int_I \tilde F(w'(s)) ds.
\end{aligned}
	\ee
	Here $\e, \delta, \ell>0$ are parameters, the double well potential reads $\tilde F(\xi):= (\xi^2-\ell^2)^2$,
	$$\kappa(w):=\frac{w''}{(1+|w'|^2)^{3/2}}$$
	is the curvature, $\rho$ is a given kernel and $\zeps(s,t) = \zp(s,t)$ for $s \in I$ and for negative times,
	with a given past data $\zp$.
	The $\tilde F$-term in \eqref{energy-kappa} is a penalization taking into account the limited
	extensibility of the filament, whereas
	the $\kappa$-term in \eqref{energy-kappa} accounts for the bending energy.
	In this paper we shall actually consider the simplified case where $\kappa(w)$ is replaced by $w''$.
	We will also assume $\delta=1$ and (with no loss of generality up to proper normalization) $\ell=1$.
	This leads to the energy functional
	\be{energy0}
	\begin{aligned}
		\cE_t (w) := \frac{1}{2 \e}& \int_I \intrp  \bigl( w(s) - \zeps(s,t-\e a)\bigr)^2 \kernel(s,a) da ds  +
		\frac12 \int_I |w ''(s)|^2 ds + \int_I F(w'(s)) ds.
	\end{aligned}
	\ee
	 Here and throughout this paper, $F$ denotes the double-well potential
	$$
	F(\xi):= (\xi^2-1)^2,\quad\hbox{with } F'(\xi)=4\xi(\xi^2-1).
	$$
	Although this will to some extent simplify the treatment, 
	this captures the essential mathematical features of the problem under study, namely 
	the conjunction of a fourth-order operator, a time delay operator and an Allen-Cahn type nonlinearity 
	acting on the space derivative.
	Here we will focus on the vanishing memory limit $\e\to 0$,
	leaving the inextensible limit $\delta\to 0$ for future investigations (in \cite{Oelz.CMS} another penalty
		method on the constraint was performed 
		but with a classical time derivative instead of our non-local time operator, 
		 and the techniques there do not apply here).
		As for the case \eqref{energy-kappa}, it presents lots of additional difficulties due to the fully nonlinear nature of the principal elliptic part, and it is not at all clear at this point if our techniques would allow to carry out such a complete analysis as for \eqref{energy0}. In view of  the novelty of problem \eqref{energy0} and of the already serious mathematical difficulties arising therein,
		we have thus preferred to leave the case \eqref{energy-kappa} for future study. Finally,
	the even more difficult case when the filament is described by a curve which is not a graph
	seems completely out of reach for the moment.

	Given the energy in \eqref{energy0}, we start from the formal variational minimization principle:
$$
\zeps (t) := \argmin_{w\in H^2(I)} \cE_t(w),\quad t>0.
$$
Define the delay operator $\cL_\e$ by
$$\begin{aligned}
	\cL_\e [z_\e](s,t)
	&= \frac{1}{\e} \int_0^{\infty} \bigl(z_\e(s,t)-z_\e(s,t-\e a)\bigr) \kernel(s,a) da,
\end{aligned}$$
which can be seen as a time derivative with memory effect and compared to a fractional derivative
	of order greater or equal to one.
The above minimization procedure leads 
to the variational problem: find $z_\e=z_\e(s,t)$, with $z_\e({\cdot},t)$ valued in $H^2(I)$,
solving 
\be{mainweak}
\left\{\begin{aligned}
	& \int_I \cL_\e [z_\e](t) v ds+\int_I \bigl(\zeps''(t) v'' + F'(\zeps'(t)) v'\bigr)ds 
	= 0,& \forall v\in H^2(I),\ t>0,\\
	&z_\e(t)=z_p(t), & t<0.
\end{aligned}
\right.
\ee
Here and throughout the paper we omit the variable $s$ when no risk of confusion arises.
The strong form of the integral equation in the latter problem reads:
\be{mainstrong1}
\cL_\e [\zeps] + \zeps '''' - \left(F'(\zeps')\right)' = 0,\quad s\in I,\ t>0,
\ee
complemented by the corresponding natural boundary conditions:
\be{mainstrong2}
\zeps'''-F'(\zeps') =  \zeps'' = 0,\quad s\in\partial I,\ t>0.
\ee

\subsection{Assumptions and notation}
In our main results we will assume that the given kernel $\rho$ satisfies
\be{hyppbm1loc}
0\le\, \rho(s,a)\in L_a^\infty(0,\infty;H^2(I))\cap L_a^1(0,\infty;H^2(I)),
\ee
\be{eq.kernel0}
\rho\in W_a^{1,1}(0,\infty;L^\infty(I)), \quad \partial_a\rho\le 0\ \ a.e.,
\ee
\be{hyppbmloc10}
0< \mu_{\min}\le \mu(s):=\int_\R \rho(s,a) da,\quad s\in I,
\ee
\be{hyp10}
\int_I \int_{a=0}^\infty a^{3/2}\|\rho(\cdot,a)\|_{2}  da <\infty,\quad 
 \int_0^\infty a\|\rho(\cdot,a)\|_{\infty} da <\infty.
\ee
For some results, we will need the following mild coercivity assumption on the kernel:
	\be{hyprhoa}
	\int_I\int_0^\infty \rho^2(s,a)| \partial_a \rho(s,a)|^{-1} dads<\infty. 
	\ee
	Note that assumption \eqref{hyp10} is for instance satisfied if
	$0\le \rho(s,a)\le K(1+a)^{-m}$ with some constants $m>\frac{5}{2}$, $K>0$,
	and that \eqref{hyprhoa} then holds if in addition $-\partial_a\rho(s,a)\ge c(1+a)^{-4}$ with $c>0$. 

As for the past data, we will assume
\be{hyppbm1locz0}
z_p\in W^{1,\infty}(-\infty,0;H^2(I)).
\ee
In connection with the kernel, we define the useful functions
\be{Sndef2}
\rhobar(s,\tau)=\int_\tau^\infty \rho(s,a)da,
\quad
\rhostar(s)=\int_0^\infty a\rho(s,a)da=\int_{\tau=0}^\infty\rhobar(s,\tau) d\tau\ee
and, for given past data $z_p$ and all $\e\ge 0$, we define the following constants
\be{defkappae}
\kappa_\e=\int_{\tau=0}^\infty \int_Iz_p(s,-\e \tau)\rhobar(s,\tau)ds d\tau.
\ee
Throughout the paper, the $L^2(I)$  scalar product  and the $L^p(I)$ norm ($1\le p\le\infty)$ will be 
respectively denoted by $(\cdot,\cdot)$  and $\|\cdot\|_p$.
 The Sobolev-Slobodecki spaces (over $(0,T)$ or $I$, possibly fractional and vector valued)
	will be denoted by $W^{k,p}$ for $k\in[0,\infty)$, $p\in[1,\infty]$
	and we will set $H^k=W^{k,2}$.
	Spaces like $L^2(I), H^k(I)$ will be often abbreviated by $L^2, H^k$.
We will sometimes omit the subscript $\e$ when no confusion may arise.
Throughout this paper, $C$ will denote a generic positive constant
	{\it depending only on $\rho, L$.}

\smallskip

	\subsection{Steady states and limiting problem}
	First, in order to motivate our results and explain their significance, 
	we need to  introduce the stationary problem associated with \eqref{mainstrong1}
	and the limiting problem as $\e\to 0$.
	The parameter $\e$ in \eqref{mainstrong1} represents a typical time scale of the memory effect,
	which gradually ``fades out'' as $\e$ becomes smaller.
	At a {\it formal} level, 
	the limiting problem to \eqref{mainstrong1} as $\e\to 0$ is given by the following parabolic problem  without delay: 
	\be{mainstrong1-0}
	\left\{\begin{aligned}
		b(s)\partial_t z_0 + z_0'''' - \bigl(F'(z_0')\bigr)' &= 0,&s\in I,\ t>0, \\
		z_0'''-F'(z_0') =  z_0'' &= 0,&s\in\partial I,\ t>0,\\
		z_0(s,0)&=\phi(s),&s\in I,
	\end{aligned}
	\right.
	\ee
	where $b=\rhostar$ is the first moment of $\rho$, defined in \eqref{Sndef2}, and $\phi=z_p(\cdot,0)$.
	We note that \eqref{mainstrong1} and \eqref{mainstrong1-0} have the same steady states, namely the solutions of the
	stationary problem
	\be{eqstatz}
	\left\{\hskip 2mm\begin{aligned}
		Z'''' - \bigl(F'(Z')\bigr)' = 0,&\quad& s\in I, \\
		F'(Z')-Z''' = Z''= 0,&\quad&s \in \partial I.
	\end{aligned}
	\right.
	\ee
	As a preliminary to the rest of our analysis, we first give the complete classification of these steady states 
	(see Section~\ref{SecSS} for precise definition and proof).
	
	\begin{prop} \label{propeqstatw0} \hspace{2em}
		\begin{itemize}
			\item[(i)]Assume $L\le\pi/2$. Then, up to additive constants,
			problem \eqref{eqstatz} admits only the affine solutions 
			$Z_0\equiv 0$ and $Z_{\pm 1}\equiv \pm x$.

			\item[(ii)]Assume $L>\pi/2$ and let $m=\lfloor2L/\pi\rfloor\in N^*$. Beside the affine solutions, 
			up to additive constants,
			\eqref{eqstatz} admits exactly $2m$ nonaffine solutions, namely 
			$Z_2,\cdots,Z_{m+1}$, and their opposites $Z_{-i}=-Z_i$.
		\end{itemize}
	\end{prop}
	
	As for the,  more classical, parabolic problem without delay \eqref{mainstrong1-0}, 
	we have the following result
	 (see Definition~\ref{def-auxilf} and Remark~\ref{rem-auxilf} for the precise notion of solution).

	\begin{prop}\label{prop-pbm0}
		Let $b\in H^2(I)$ with $\inf_I b>0$ and $\phi\in H^2(I)$.
		
		\smallskip
		
		(i) Problem \eqref{mainstrong1-0} has a unique, global solution
		\be{regul-0}
		z_0\in  C_b([0,\infty);H^2(I))\cap H^1_{loc}([0,\infty);L^2(I))\cap L^2_{loc}([0,\infty);H^4(I)).
		\ee
		 Moreover, for each $p\in(2,\infty)$, we have
		\be{regul-01}
		 \sup_{t\ge 1} \ \left(\,\|z_0\|_{W^{1,p}(t,t+1;L^2(I))}+\|z_0\|_{L^p(t,t+1;H^4(I))}\,\right)<\infty.
		\ee
		
		(ii) There exists a stationary solution $Z_0$, with  $(Z_0,b)=(\phi,b)$, such that
		\be{cv-quasistat0-0}
		\lim_{t\to \infty} \|z_0(t)-Z_0\|_{H^2(I)}=0.
		\ee
	\end{prop}
	
	Although Proposition~\ref{prop-pbm0} may be known, we have been unable to find a suitable reference,
		especially in view of the needed time regularity, in presence of nonlinear 
		boundary conditions.
		We thus provide a proof in  \ref{Secz0}, 
		based on maximal regularity estimates from \cite{DHP} for inhomogeneous, linear higher order 
		parabolic problems.
	
	\subsection{Main results}
	
		\vskip 1pt
		
	The main goals of this work are to show that:
	
	\vskip 2pt
	
	$\bullet$ For sufficiently small $\e>0$,
	the time-nonlocal problem \eqref{mainstrong1} is globally well posed
	and its solution converges to one of the steady states as $t\to\infty$.
	
	\vskip 2pt
	
	$\bullet$ When $\e\to 0$, the solution $z_\e$ of \eqref{mainstrong1} converges 
	in a certain sense to the solution $z_0$ of 
	problem \eqref{mainstrong1-0}.
	\vskip 2pt
	
	$\bullet$ For sufficiently small $\e>0$,
	problem \eqref{mainstrong1} inherits part of the large time asymptotic properties of \eqref{mainstrong1-0}.
	
	\vskip 2pt
	Although this may seem reasonable, this is by no means obvious, in view of the nonlocal nature of \eqref{mainstrong1},
	and will require rather involved arguments.

	\goodbreak

	\begin{thm} \label{thmglob}
		Assume \eqref{hyppbm1loc}--\eqref{hyp10} and \eqref{hyppbm1locz0}.
		There exists $\e_0>0$ depending only on $\rho,  L$ and on
		$\|z_p\|_{L^\infty(-\infty,0;H^2(I))}$ such that,
		for all $\e\in(0,\e_0]$, the following holds:
		
		\smallskip
		(i) Problem \eqref{mainweak} admits a unique global solution $z_\e\in L^\infty(\R;H^2(I))$.
		
		\smallskip
		(ii) The curve $\e \mapsto z_\e(t)$ is locally $1/4$-H\"older continuous from $(0,\e_0]$ into $H^2(I)$, 
		uniformly for $t>0$ in bounded intervals.
		\smallskip
		
		(iii) There exists a stationary solution $Z_\e$, with $(Z_\e,\rhostar)=\kappa_\e$, such that,	
		\be{convH2b}
		z_\e(t)\to Z_\e \hbox{ in $H^2(I)$, as $t\to\infty$.}
		\ee
		
		(iv) The set of $\e$ such that $Z'_\e\equiv 1$ (resp.~$-1$) is relatively open in $(0,\e_0]$.
	\end{thm}
	
	\goodbreak

	\begin{thm}\label{thm-cvE}
		Assume \eqref{hyppbm1loc}--\eqref{hyp10}, \eqref{hyppbm1locz0}, 
		and let $z_0$ be the solution of \eqref{mainstrong1-0}  with $b=\rhostar$ and $\phi=z_p(s,0)$.
		
		\smallskip
		
		(i) For every fixed $T>0$, we have
		\be{cvzeps}
		\lim_{\e\to 0}\zeps = z_0,
		\ee
		where the convergence is strong in $C([0,T];C^1(\bar I))$, weak in $H^1(0,T;L^2(I))$ and weak-${}^*$ in $L^\infty(0,T;H^2(I))$.
		If in addition \eqref{hyprhoa} holds,
		then the convergence in \eqref{cvzeps} is strong in $L^2(0,T;H^2(I))$.
		\smallskip
		
		(ii) Assume that \eqref{hyprhoa} holds and that $Z'_0\equiv 1$ (resp.~$-1$).
		Then there exists $\bar\e_0\in(0,\e_0]$ such that, for all $\e\in(0,\bar\e_0]$, $Z'_\e\equiv 1$  (resp.~$-1$).
	\end{thm}

\begin{rem} \label{rem-mainth}
	 (i) 
		By a simple energy argument  (see Remark~\ref{rem-pbm0stab}(i) for details), the stable steady states of \eqref{mainstrong1-0} 
		are exactly the nonconstant affine steady states ($Z\equiv \pm x+Const.$).
		More precisely, if $ \phi\in H^2(I)$ and $\|\phi''\|_2+\|\phi'\pm 1\|_\infty$ is sufficiently small, 
		then $Z'_0\equiv \mp 1$.
		Theorem~\ref{thm-cvE}(ii) guarantees that under
		this assumption for $\phi=z_p(0)$, we have $Z'_\e\equiv -1$ (resp.~$1$) for $\e>0$ small.
		
		On the contrary, the result of Theorem~\ref{thm-cvE}(ii) is not expected to hold if $Z'_0\not\equiv\pm 1$,
		since the steady state is then unstable even for the limiting problem (cf.~Remark~\ref{rem-pbm0stab}(ii)).
	
	\smallskip
	
	(ii) In Theorem~\ref{thmglob}, by $z_\e\in L^\infty(\R;H^2(I))$ being a solution we mean that the first (resp., second) part of
	\eqref{mainweak} is satisfied for a.e.~$t>0$ (resp., $t<0$).
	Under the assumptions of Theorem~\ref{thmglob},
	the solution actually enjoys the additional regularity
	\be{regulhigher}
	z_\e\in L^\infty(0,\infty;W^{4,\infty}(I))\cap W^{1,\infty}(0,\infty;H^2(I)),
	\ee
	so that the first part of \eqref{mainweak} is satisfied for all $t>0$
	and $z_\e$ is in fact a strong solution.
	Moreover, it satisfies the uniform global bound
	$$\|z_\e(t)\|_{H^2(I)}\le C\bar R,\quad t>0,\ \e\in (0,\e_0],$$
	and the uniform, time derivative global integrability property
		\begin{equation}\label{eq.stab.l2.tps.espace}
			\int_0^\infty\int_I |\partial_t z_\e|^2 dsdt \le C\bar R,\quad \e\in (0,\e_0],
		\end{equation}
		where
	$$R:=1+\|z_p\|_{L^\infty(-\infty,0;H^2(I))},\quad \hat R:=\|z_p\|_{W^{1,\infty}(-\infty,0;L^\infty(I))}, \quad
	\bar R=R^4+\hat R^2.$$
	Also we may take 
	$\e_0=c_0\bar R^{-4}$,  with $c_0=c_0(\rho, L)>0$.
	See Propositions~\ref{thmregtime} and \ref{thmunifbounds}. 
	
	\smallskip
	
	(iii) The solution $z_\e$ is continuous in $H^2(I)$ for $t>0$  and has a limit as $t\to 0^+$ (owing to \eqref{regulhigher}),
	but it generally has 
	a jump discontinuity at $t=0$, unless a suitable  
	compatibility condition is imposed on $z_p$.
	A useful estimate of the jump, of independent interest, is given in Proposition~\ref{lem.discont}.
\end{rem}

\subsection{Organisation and main ideas of the proofs}
In Section~\ref{SecSS} we describe the steady states of the problem by means of ODE analysis,
providing results that will be used in Section~\ref{SecConv} for the proof of the asymptotic behavior as $t\to\infty$ (Theorem \ref{thmglob}(iii)).

 Local existence and uniqueness  (Theorem \ref{thmloc} below), as well as time regularity 
(Proposition~\ref{thmregtime}), rely on a Banach fixed point. 
 The proofs, given in Section~\ref{SecLoc}, are based on fixed point arguments
	built on the resolution of a fourth-order elliptic
	problem. It is to be noticed that this is not standard in the framework of partial 
	differential evolution equations where a parabolic problem is more often used.
	Here we fix-point not only the nonlinearity but also the past of the solution,
	 involved as a
	source term in the fixed-point operator. 
	 The fourth-order 
	elliptic problem is analyzed in Section~\ref{SecRes}, where we provide resolvent estimates in terms of $\e$ that are required
	for the local existence-uniqueness. This is done in the Hilbert setting and for classical solutions,
	giving by interpolation a general result in $L^p$ spaces for $p\geq 2$ which 
	is of interest {\em per se}  (cf.~Propositions \ref{lemresolvent0.bis} and \ref{lemresolvent.bis}).

	The previous step guarantees a local time of existence for $\e$ small enough.
		In order to extend this time,
	 so as to prove the global existence part of Theorem \ref{thmglob}(i), we  first show in Section~\ref{SecEn}
	a stability result,
	 namely the time derivative global integrability property
	\eqref{eq.stab.l2.tps.espace} (see Proposition~\ref{thm.nrj}(ii)). One of the main features of this estimate is the uniformity with respect to $t$ 
	and to $\e$,  which plays a key role in subsequent proofs. Its proof is much more involved than for \eqref{mainstrong1-0}. Indeed, for   smooth solutions of classical
	gradient flow models one has easily that:
	\begin{equation}\label{edo.energie.zz}
		\ddt{} \left\{ \int_0^L F(\zz'(t))ds + \frac12\|\zz''(t)\|_2^2 \right\}=- \int_0^L  b(s)\left|\dt \zz(t) \right|^2 ds,
	\end{equation}
	giving directly the result after integration in time. For delayed problems
	this  does not hold and an alternative approach has to be found. 
	Thanks to the monotonicity condition \eqref{eq.kernel0}, one first ensures 
	 the energy dissipation property
	\begin{equation}\label{edo.energie.zeps}
		\ddt{} \cE_t(\zeps) = \frac{1}{2\e^2} \int_0^L \int_0^\infty  \bigl(\zeps(t)-\zeps(t-\e a)\bigr)^2\partial_a \kernel \, da ds <0,
	\end{equation}
	which guarantees the time integrability of the 
	right hand side of \eqref{edo.energie.zeps} (see Proposition~\ref{thm.nrj}(i)). This latter
	term is then used as a source term for a closed elliptic problem 
	satisfied by $\dt \zeps$  (Lemma~\ref{lem.dt.z.unif2})
	and from this we deduce an $L^\infty_t L^2_s$ bound for $\zeps'$ and $\zeps''$
	(see Proposition~\ref{thm.dt.z.unif}).
	
	Since our basic working space for local well-posedness is $L^\infty_t H^2_s$ and our boundary conditions are of (nonlinear) Neumann type, 
	this is still unsufficient to conclude global existence and it remains to bound $\zeps$ itself.
	This is achieved in a second step in Section~\ref{SecUnif},
	where a uniform $L^\infty$ bound in time and space
	 is derived (Proposition~\ref{thmunifbounds}), based on a new invariant  
	\begin{equation}\label{invariant.intro}
	\int_0^L \int_0^\infty \zeps(s,t-\e a)\varphi(s,a) da ds=\kappa_\e,
		\end{equation}
	where $\varphi$ is defined in \eqref{Sndef2},
	which appears herein as a key quantity
	 (its properties are derived in Lemma~\ref{lemunifbounds}).
	While $\varphi$ is somehow reminiscent of the so-called left eigenvector 
	in studies of hyperbolic equations from population dynamics 
		\cite[Chap. 3]{Perth.Book}, it is completely new in the present context.
	This invariant is to be related to $(\mu_1,\zz)$ to which it formally reduces when $\e$ goes to zero. 
	As a matter of fact, we show, cf.~Theorem \ref{thmglob}{\em (iii)}, that
	actually for $\e>0$ fixed, $\zeps(t) \to Z_\e$ when $t$ grows large
	and $(Z_\e,\mu_1) = \int_0^L \intrp \zp(-\e a) \varphi(s,a) da ds$. The same result 
	holds as well when $\e=0$ and is more standard (cf.~Proposition \ref{prop-pbm0}{\em (ii)}).
	
	  In Section~\ref{SecCont}, the (H\"older) continuity with respect to $\e$ in Theorem \ref{thmglob}{\em (ii)} is 
	 established (cf.~Proposition~\ref{prop-conteps}) by making use in particular of the above energy estimates.
	
	Next, in Section~\ref{SecConv}, Theorem \ref{thmglob}{\em (iii)}, i.e.~the convergence to a steady state as $t\to\infty$ for fixed $\e$, is proved by dynamical systems arguments
	(cf.~Lemmas~\ref{prop-quasistat} and \ref{lem-quasistat}),
	based on the Liapunov functional given by the energy and on the special structure of the set of equilibria.
	Namely, there is a finite number of equilibria up to additive constants (see Proposition~\ref{propeqstatw2}) and, whereas the energy itself 
	does not discriminate the additive constant, the stabilization to a single value of the 
	additive constant follows from the existence of the invariant in \eqref{invariant.intro}.

	 We next prove Theorem~\ref{thm-cvE}(i), 
	 i.e.~the convergence of $\zeps$ to the solution $z_0$ of the parabolic problem \eqref{mainstrong1-0}
	as $\e\to 0$.
	This is first established in a weak $H^2_s$ sense in Section~\ref{SecConvEps}, and then in the strong $L^2_t H^2_s$ sense in~Section~\ref{SecStab},
	under the additional coercivity assumption \eqref{hyprhoa} on the kernel $\rho$.
	The proofs rely on interpolation and compactness arguments which make use of the 
	{\em a priori} estimates on $\dt \zeps$ mentioned above,  plus some additional higher order {\em a priori} estimates (Lemma~\ref{lem.compact0}).
	The strong convergence result is new as compared with 
	\cite{MiOel.2,MiOel.4,Mi.5}, to which 
	it applies as well (see Remark \ref{rmk.10.1}). Moreover it allows to show Theorem~\ref{thm-cvE}(ii),
	namely the stability of the affine stationary states $Z'= \pm 1$. This relies on three main ingredients:
	(i)  these stationary states are global minima of the mechanical energy 
	$E_0(Z):=\int_0^L (\frac12|Z''|^2+F(Z')) ds$;  
	(ii)  the convergence of the energy $E_\e(z_\e)$ to $E_0(z_0)$ (as a consequence of the above $L^2_t H^2_s$-convergence);
	(iii) the decrease of energies in time.
	
	 Finally, in the appendix, we provide a proof of the properties of the  parabolic problem~\eqref{mainstrong1-0}:
	  existence-uniqueness-regularity (Proposition \ref{prop-pbm0}(i)),
	energy balance (Proposition~\ref{prop-pbm0bEn}),   and stabilization (Proposition \ref{prop-pbm0}(ii)).

\medskip


\section{Steady states}\label{SecSS}

\newcommand{\sn}{{\mathrm{sn}}}
Before entering the analysis of the evolution problem \eqref{mainweak}, 
a basic and more elementary task is to describe the steady states and to establish Proposition~\ref{propeqstatw0},
which is needed in the proof of the convergence part of Theorem~\ref{thmglob}.

We say that $Z$ is a (classical) solution of the stationary problem \eqref{eqstatz} if $Z\in C^4([0,1])$ and $Z$ satisfies \eqref{eqstatz} pointwise.
It is not difficult to see that weak and classical solutions of \eqref{eqstatz} are equivalent notions.
Indeed, if $Z\in H^2(I)$ is a weak solution, it solves:
$$
\left(Z'',v''\right) + (F'(Z'),v') =0,\quad \forall v \in H^2(I)
$$
then it solves \eqref{eqstatz} in the distributional
sense. As $Z' \in C(\overline{I})$ and $Z'' \in L^2(I)$, this implies that $Z'''' \in L^2(I)$ and
thus $Z \in H^4(I) \subset C^3(\overline{I})$ thanks to Sobolev's embeddings. 
Again using $\eqref{eqstatz}$ provides $Z'''' \in C^1(\overline{I})$.
The sufficient condition
is obvious.

To prove Proposition~\ref{propeqstatw0} (and give additional information on the solutions), we proceed as follows.
Observe that $z$ is a solution of \eqref{eqstatz}
if and only if $w:=z'$ solves
\be{eqstatw}
w''=F'(w)=-4w(1-w^2) \ \ \hbox{ in $(0,L)$, \quad with }w'(0)=w'(L)=0.
\ee
To study \eqref{eqstatw}, for all $a\in \R$, we consider the shooting problem
\be{eqstatw2}
w''=F'(w) \quad \hbox{ with $w'(0)=0$ and $w(0)=a$.}
\ee
Problem \eqref{eqstatw2} has a unique solution $w_a$, defined on a maximal time interval $[0,S_a^*)$
for some $S_a^*\in(0,\infty]$.
We clearly have $w_0=0$, $w_{\pm 1}=\pm 1$ (constant solutions).
Also by symmetry, we have $w_{-a}=-w_a$, so that it thus suffices to consider $a>0$, $a\ne 1$.
This case is treated in the following proposition.

\begin{prop} \label{propeqstatw} \hspace{1em}
	\begin{itemize}
		\item[(i)]If $a>1$ then $w'_a>0$ on $(0,S_a^*)$. Moreover, $S_a^*<\infty$ and $\lim_{s\to S_a^*} w'_a(s)=\infty$.
		
		\item[(ii)]For $a\in(0,1)$, we have $S_a^*=\infty$ and $w_a$ is given by 
		 an anti-periodic solution of $w''=F'(w)$, for some half-period $\Pi_a>0$
		(which is in particular a periodic solution of period $2\Pi_a$) and is explicit:
		\begin{equation}\label{eq.exact.sol}
			w_a(s) = a \, \sn\left(\sqrt{2(2-a^2)}s + K(k_a),k_a\right), \quad k_a = \frac{a}{\sqrt{2-a^2}},
		\end{equation}
		where $\sn$ denotes  the {\em sine amplitude elliptic function} and $K(k)$ the {\em complete elliptic integral of the first kind.} 
		The number $K(k)$ is a quarter period associated to $\sn(\cdot,k)$ and reads:
		$$
		K(k):= \int_0^{\frac{\pi}{2}} \frac{1}{\sqrt{1-k^2 \sin^2\hskip -1.5pt \theta}} \,d\theta,\quad k \in (0,1).
		$$
	Moreover we have $w'_a< 0$ on $(0, \Pi_a)$.
	
		\item[(iii)]The half-period function $(0,1)\ni a\mapsto \Pi_a$ is continuous monotone increasing with 
		\be{steady0}
		\lim_{a\to 0^+} \Pi_a= \pi/2,\qquad \lim_{a\to 1} \Pi_a=\infty 
		\ee
		and reads: $\Pi_a := K(k_a)/\sqrt{2(2-a^2)}$.
	\end{itemize}
\end{prop}

As a consequence we can describe the solutions of \eqref{eqstatw} and \eqref{eqstatz} as follows.

\begin{prop} \label{propeqstatw2} \hspace{1em}
	\begin{itemize}
		\item[(i)]Assume $L\le\pi/2$. Then problem \eqref{eqstatw} admits only the  constant solutions $0$ and $\pm 1$.
		\item[(ii)]Assume $L>\pi/2$ and let $m=\lfloor 2L/\pi\rfloor\in N^*$. Beside the constant solutions, \eqref{eqstatw} admits an even finite number $2m$ of solutions $\pm w_1,\cdots,w_m$
		where, for each $n\in\{1,\dots,m\}$, $w_n$ is the unique antiperiodic solution of
		\eqref{eqstatw2} with $a>0$ such that $\Pi_a=L/n$.
		\item[(iii)]The solutions of \eqref{eqstatz} are given by $z(s)=\int_0^s w(\tau)d\tau + c$ where $w$ is any solution of 
		\eqref{eqstatw} and $c\in\R$ is arbitrary.
	\end{itemize}
\end{prop}

\begin{proof}[Proof of Proposition~\ref{propeqstatw}]
	We shall denote $w=w_a$ for simplicity.
	
	\smallskip
	
	(i)
	If $a>1$ then $w$ cannot be a solution of \eqref{eqstatw}.
	Indeed, we have $w''(0)>0$ and we easily deduce that $w''$ and $w'$ remain $>0$ for $s>0$ as long as $w$ exists.  It follows that $w''\ge cw^3$ for all $s\in (0,S^*)$ with some $c>0$,
		and a standard argument then implies $S^*<\infty$ and $\lim_{s\to S^*} w'(s)=\infty$.

	\smallskip (ii) Multiplying \eqref{eqstatw2} by $w'(s)$ gives
	$
	\nud \left( w_a'\right)^2 {\phantom{}}'= F(w_a)'
	$. We integrate with respect to $s$. This leads to:
	$$
	\nud \left( w_a'\right)^2(s) - F(w_a(s)) = \nud \left( w_a'\right)^2(0) - F(w_a(0)).
	$$
	Now using the boundary conditions in \eqref{eqstatw2} provides:
	$$
	\nud \left( w_a'(s)\right)^2 = F(w_a(s)) - F(a),
	$$
	which finally gives:
	$$
	\left( w_a'(s)\right)^2 = 2(1-F(a) )-4 w_a(s)^2 + 2 w_a(s)^4 =: A w_a^4 + B w_a^2 +  c.
	$$
	We compare this equation with $\sn(s,k)$ satisfying \cite{Schwalm.Book}:
	$$\left(  \sn'(s,k) \right)^{2}=\alpha\, \sn^{4}(s,k)+\beta\, \sn^{2}(s,k)+\gamma, \quad \alpha := k^2, \quad \beta := -(1+k^2),\quad \gamma =1,$$
	where $k$ is a given parameter in $(0,1)$.
	We look for $w_a$ to be  of the form:
	$$
	w(s) = \xi \,\sn (\chi  \,(s-s_0),k),
	$$
	where $\xi$, $\chi$ and $k$ are constants to be found.
	This provides a system of 3 unknowns s.t.
	$$
	2 = \frac{\chi^2}{\xi^2} k^2 ,\quad 4=\chi^2 (1+k^2) , \quad 2(1-F(a)) = \xi^2 \chi^2.
	$$
	After some computations and accounting that $k\in(0,1) $, one obtains the following definition of the constants $\xi,\chi,k$ solving the previous system:
	$$
	\xi_a=a,\quad \chi_a= \sqrt{2(2-a^2)}, \quad k_a = 
	\frac{a}{\sqrt{2-a^2}}.
	$$
	We  next ensure that the solution satisfies the boundary conditions \eqref{eqstatw2}
	at the origin $s=0$. All this leads to the explicit form \eqref{eq.exact.sol}.
	It is a periodic solution of period $4K(k_a)/\sqrt{2(2-a^2)}$.

	\smallskip
	
	(iii) The complete elliptic integral of the first kind $K(k)$ is a regular monotone increasing function of $k \in (0,1)$ s.t. 
	(cf.~\cite{Schwalm.Book}), 
	$\lim_{k \to 0}  K(k) = \frac{\pi}{2}$, and $\lim_{k \to 1} = K(k) = + \infty$, so that $\Pi_a = 2 \frac{K(k_a)}{\chi_a}$
	and one has
	$$
	\lim_{a \to 0} \Pi_a = \frac{\pi}{2} , \quad  \lim_{a \to 1} \Pi_a = +\infty.
	$$
	Moreover, 
	$$
	\Pi_a'= - \int_0^{\frac{\pi}{2}} \frac{4 a (\cos^2(t)-2)dt}{\left({4 (1- a^2) + 2 a^2\cos^2(t)}\right)^{\frac{3}{2}}}  >0, \quad \forall a \in (0,1),
	$$
	which ends the proof.
\end{proof}

\begin{proof}[Proof of Proposition~\ref{propeqstatw2}]
	In view of assertion (i) of Proposition~\ref{propeqstatw}, if $a>1$ then $w$ cannot be a solution of \eqref{eqstatw}.
	In view of assertion (ii), for $a\in (0,1)$, the function $w_a$ is a solution of \eqref{eqstatw} (i.e.,~solves the boundary condition 
	$w'(L)=0$) if and only if its antiperiod is a submultiple of $L$, that is, $\Pi_a=L/n$ for some integer $n\ge 1$. 
	This implies that:
	$$
	L \geq \Pi_a \geq  \frac{\pi}{2}.
	$$
	We exclude the case $L=\pi/2$ since then no non-trivial solution $w_a$ exists.
	Next, the problem: find $(a_n)_{n \in \{1,\dots,m\}} \in (0,1)^m$ s.t.
	$$
	\frac{L}{n}=  \Pi_{a_n}
	$$
	admits a unique solution iff $L/n \in [\pi/2,\infty)$ for all $n\in\{1,\dots,m\}$, which is precisely the way $m$ was chosen.
	This procedure guarantees the construction 
	 of $2m$ distinct solutions of \eqref{eqstatw}. 
	
	This along with the constant solutions describes the set of solutions of \eqref{eqstatw} and completes the proof of assertions (i) and (ii). Assertion (iii) follows immediately by integration.
\end{proof}
\section{Resolvent operator}\label{SecRes}

In order to first establish the local in time existence-uniqueness (by a suitable fixed point argument), problem 
\eqref{mainstrong1}-\eqref{mainstrong2} will be rewritten under the form
\be{pbmloc}
\left\{\hskip 2mm\begin{aligned}
	(\mu(s)+\e\partial^4_s)z&=\e\partial_s(F'(\partial_sz))+\int_0^\infty z(s,t-\e a)\rho(s,a) da, 
	&\; & 0<t<T,\ s\in I, \\
	\partial^2_s z(s,t)&=\partial^3_s z(s,t)-F'(\partial_sz(s,t))=0 &\; &0<t<T,\ s\in\partial I, \\
	z(t)&=z_p(t) &\; &t<0.
\end{aligned}
\right.
\ee
To study \eqref{pbmloc}, for given $\e>0$, we will need resolvent estimates for the linear auxiliary problem:
\be{linearpbm2}
\left\{
\begin{aligned}
	&	ku+\e u''''=\e f'+g,& s\in I, \\
	&	u''=0,\, u'''=f, & s\in\partial I.
\end{aligned}
\right.
\ee
	Here $k_2\ge k_1>0$ and $k=k(s)\in L^\infty(I)$ is a fixed function such that 
	\be{hyplinearpbm}
	0<k_1\le k(s)\le k_2,\quad s\in I.
	\ee
	In this section, $K$ denotes a generic positive constant depending only on $k_1, k_2, L$.
	For given data $(f,g)$ in $L^2(I)\times L^2(I)$,
	by a weak solution of \eqref{linearpbm2} we understand a function $u\in H^2(I)$ which satisfies the variational identity
	\be{linearpbmweak}
	\int_I k(s)u(s)v(s)ds+\e\int_I u''(s)v''(s) ds=-\e \int_I f(s)v'(s)ds+\int_I g(s)v(s)ds,\quad \forall v\in H^2(I).
	\ee
	
	We start with the basic existence-uniqueness result for \eqref{linearpbm2}.

	\begin{prop} \label{lemresolvent0.bis}
		Assume \eqref{hyplinearpbm} and let $\e>0$,  $(f,g)\in L^2(I)\times L^2(I)$.
		\begin{itemize}
			\item[(i)]Problem \eqref{linearpbm2} admits a unique weak solution.
			
			\smallskip
			
			\item[(ii)]Assume in addition that $f\in H^1(I)$. Then a function $u\in H^2(I)$ is a weak solution if and only if it is a strong solution
			i.e.,~$u\in H^4(I)$, $u$ satisfies the equation for a.e.~$s\in I$ and the boundary conditions pointwise.
		\end{itemize}
	\end{prop} 
	
	We shall denote by $A_\e:L^2(I)\times L^2(I)\to H^2(I)$ the resolvent operator, defined by $(f,g)\mapsto u$.
	Note that, by linearity, we may write
	\be{defA1A2a}
	A_\e(f,g)=A_{1,\e}f+A_{2,\e}g,
	\ee
	where
	\be{defA1A2}
	A_{1,\e}:f\in L^2(I)\mapsto A_\e(f,0), \quad A_{2,\e}:g\in L^2(I)\mapsto A_\e(0,g).
	\ee
	
	The following proposition provides the required resolvent estimates.
	
	\begin{prop} \label{lemresolvent.bis}
		Assume \eqref{hyplinearpbm} and let $\e\in(0,1)$. 	\smallskip
		
		(i) Let $(f,g)\in L^2(I)\times L^2(I)$. Then $u=A_\e(f,g)$ satisfies
		\be{EstimH2}
		\nrm{u^{(j)}}{2}\leq K \e^{-j/4}\bigl(\e^{3/4} \nrm{f}{2}+\nrm{g}{2}\bigr),\quad j\in\{0,1,2\}.
		\ee
		
		(ii) Let $q\in[2,\infty]$ and $(f,g)\in W^{1,q}(I)\times L^q(I)$. Then,
		for any integer $j\in[0,4]$, we have
		\be{res1.bis}
		\|A_{1,\e}f\|_{W^{j,q}}\le  K \e^{(3-j)/4}\left(\|f\|_q+\e^{1/4}\|f'\|_q\right)
		\ee
		and
		\be{res1.bis2}
		\|A_{2,\e}g\|_{W^{j,q}}\le  K \e^{-j/4}\|g\|_q.
		\ee
		\smallskip
		
		(iii) 
	Let  $K_1>0$ and assume that 
			\be{hypK1}
			k\in H^2(I),\quad \|k\|_{H^2(I)}\le K_1.
			\ee
			Then $A_{2,\e}\in {\mathcal{L}}(H^2(I), H^4(I))$ and
			\be{res2.bis}
			\|A_{2,\e}\|_{\mathcal{L}(H^2(I), H^{2+j}(I))}\le  \tilde K \e^{-j/4},\quad j\in\{0,1,2\},
			\ee
			where the constant $\tilde K$ depends only on $k_1,k_2,K_1,L$.
	\end{prop}

	\begin{proof}[Proof of Proposition~\ref{lemresolvent0.bis}]
		(i) This is an immediate consequence of Lax-Milgram's Theorem. 
		
		(ii) The weak solution $u$ satisfies the differential equation in \eqref{linearpbm2} in the sense of distributions and,
		assuming $f\in H^1(I)$, we have
		$u''''=\e^{-1}(g+\e f'-k u)\in L^2(I)$,
		so that $u\in H^4(I)\subset C^3(\overline J)$.
		As one has enough regularity, starting from the weak formulation and integrating by parts one shows that
		$$
		\bigl[(u'''-f)v-u''v'\bigr]_0^1 	=\int_I \bigl(u''''-f'+\e^{-1}(k u- g)\bigr)v=0
		\quad\hbox{ for all $v\in H^2(I)$.}
		$$
		It follows easily that $u'''=f$ and $u''=0$ on $\partial I$.
		
		Conversely, if $f\in H^1(I)$ and $u$ a strong solution, by multiplying by $v\in H^2(I)$ and integrating by parts, we easily obtain 
		\eqref{linearpbmweak}.
	\end{proof}

	\begin{proof}[Proof of Proposition~\ref{lemresolvent.bis}]
		(i) Recall the following interpolation inequality (which is a consequence of \cite[Theorem 7.37, p.198]{Leoni.Book}):
		If $|J|\ge L_0>0$, there exists $K=K(L_0)>0$ such that, for all $\eta>0$
		\be{interpol-eta}
		\|u'\|_{L^2(J)}\le \eta\|u''\|_{L^2(J)}+ K(1+\eta^{-1})\|u\|_{L^2(J)}.
		\ee
		For $\delta_1,\delta_2,\eta>0$ to be fixed, the choice $v=u$ in \eqref{linearpbmweak} yields
		$$\begin{aligned}	
			\int_I &\bigl(ku^2+\e{u''}^2\bigr)(s) ds 
			\leq   \delta_1  \e \nrm{u'}{L^2}^2 + \delta_1^{-1}\e\nrm{f}{2}^2 + \delta_2 \nrm{u}{2}^2+ \delta_2^{-1} \nrm{g}{2}^2 \\
			&\leq   \delta_1  \e\eta \nrm{u''}{L^2}^2+(C_0(1+\eta^{-1})\delta_1 \e+\delta_2)\nrm{u}{L^2}^2+ \delta_1^{-1}\e\nrm{f}{2}^2 + \delta_2^{-1} \nrm{g}{2}^2,
		\end{aligned}$$
		with $C_0=C_0(k_1,k_2)>0$.
		Choosing $\delta_1=\min(\frac12,\frac{k_1}{6 C_0})\e^{-1/2}$, $\eta=(2\delta_1)^{-1}$
		and $\delta_2=\frac{k_1}{3}$, we get
		$$C_0(1+\eta^{-1})\delta_1 \e\le C_0(1+2\delta_1)\frac{k_1}{6 C_0}\e^{1/2}
		\le 2C_0\e^{-1/2}\frac{k_1}{6 C_0}\e^{1/2}=\frac{k_1}{3},
		$$
		so that
		$$k_1\nrm{u}{L^2}^2+\e\nrm{u''}{L^2}^2
		\leq  \frac{\e}{2} \nrm{u''}{L^2}^2+\frac{2k_1}{3}\nrm{u}{L^2}^2+\frac{3}{k_1}\nrm{g}{2}^2 +
		2\e^{3/2}\nrm{f}{2}^2,$$
		hence \eqref{EstimH2} for $j\in\{0,2\}$. The case $j=1$ follows from \eqref{interpol-eta} with $\eta=\e^{1/4}$.

		\smallskip
		(ii) First consider the case $q=2$. For $j\in\{0,1,2\}$, \eqref{res1.bis} and \eqref{res1.bis2} follow from assertion~(i). 
		For $j=4$ it is then a consequence of $u''''=\e^{-1}(g-k u)+f'$. The result for $j=3$ 
		is then obtained by interpolating through \eqref{interpol-eta}.
		
		We next consider the case $q=\infty$ and establish the $L^\infty$ estimate by means of a contradiction and rescaling argument.
		Thus assume that \eqref{res1.bis} or \eqref{res1.bis2} with $j=0$ fails. Then there exist sequences $\e_i\in (0,1)$, 
		$k_i, f_i\in L^\infty(I)$, 
		$g_i\in W^{1,\infty}(I)$ such that,  denoting by $u_i$ the corresponding solutions, we have
		$$k_1\le k_i(s)\le k_2,\quad \|g_i\|_\infty+\e_i^{3/4}\|f_i\|_\infty+\e_i\|f_i'\|_\infty=1,\quad  
		M_i:=\|u_i\|_\infty\to \infty.$$
		Pick $s_i\in \overline I$ such that $|u_i(s_i)|=M_i$ and define the rescaled functions
		$$v_i(\sigma)=M_i^{-1}u_i(s_i+\e_i^{1/4}\sigma),\quad \sigma\in 
		J_i=[a_i,b_i]:=[-\e_i^{-1/4}s_i,\e_i^{-1/4}(L-s_i)].$$ 
		Since $u_i$ is a strong solution by Proposition~\ref{lemresolvent0.bis}(ii), we have $v_i\in W^{4,\infty}(J_i)$ and $v_i$ satisfies
		\be{linearpbm}
		\left\{\hskip 2mm\begin{aligned}
			&\tilde k_i v_i+v_i''''=\e_i^{3/4}\tilde f_i'+\tilde g_i,&\quad& a.e.\ \sigma\in J_i, \\
			&v_i''=0,\; v_i'''=\e_i^{3/4}\tilde{f}_i, &\quad& \sigma \in\partial J_i,
		\end{aligned}
		\right.
		\ee
		with $\tilde k_i(\sigma)=k_i(s_i+\e_i^{1/4}\sigma)$, $\tilde f_i(\sigma)=M_i^{-1}f_i(s_i+\e_i^{1/4}\sigma)$
		and $\tilde g_i(\sigma) := M^{-1}_ig_i(s_i+\e_i^{1/4}\sigma)$.
		Then $\|v_i\|_{L^\infty(J_i)}=|v_i(0)|=1$ and
		\be{w-W4a}
		\begin{aligned}
			\e_i^{3/4}\|\tilde f_i'\|_{L^\infty(J_i)}
			&+\e_i^{3/4}\|\tilde f_i\|_{L^\infty(J_i)}+\|\tilde g_i\|_{L^\infty(J_i)} \\
			&=M_i^{-1}\bigl(\e_i\|f_i'\|_{L^\infty(I)}+\e_i^{3/4}\|f_i\|_{L^\infty(I)}+\|g_i\|_{L^\infty(I)}\bigr)\le M_i^{-1}.
		\end{aligned}
		\ee
		In particular, $\|v_i''''\|_{L^\infty(J_i)}\le c$. Here and in the rest of the proof, $c$ denotes a 
			generic positive constant independent of $i$. We deduce
		\be{w-W4}
		\|v_i\|_{W^{4,\infty}(J_i)}\le c
		\ee
		by interpolation
		(note that $|J_i|\ge L$ due to $\e_i<1$).
		Let $\alpha<\beta$ be such that $[\alpha,\beta]\subset J_i$. 
		 Multiplying the differential equation \eqref{linearpbm} with $v_i$ and integrating by parts, we obtain
		\be{intalpha}
		\int_\alpha^\beta \tilde k_i v_i^2+(v_i'')^2=\bigl[v_i'v_i''-v_iv_i'''\bigr]_\alpha^\beta
		+\int_\alpha^\beta (\e_i^{3/4} \tilde f_i'+\tilde g_i) v_i d \sigma .
		\ee

		$\bullet$ First consider the case when $a_i$ and $b_i$ are bounded. Applying \eqref{intalpha} with 
		$\alpha=a_i$ and $\beta=b_i$ and using \eqref{w-W4a}, \eqref{w-W4},
		along with the boundary conditions in \eqref{linearpbm}, yields
		$$\begin{aligned}
			1&=\|v_i\|_{L^\infty(J_i)}\le c\int_{a_i}^{b_i} v_i^2+(v_i'')^2 
			\le c\e_i^{3/4}\bigl|\bigl[ \tilde f_iv_i\bigr]_{a_i}^{b_i}\bigr|+c\int_{a_i}^{b_i} (\e_i^{3/4} \tilde f_i'+\tilde g_i)^2 \\
			&\le c\e_i^{3/4}\| \tilde f_i v_i\|_{L^\infty(J_i)}+(b_i-a_i)
			\bigl(\e_i^{3/4} \|\tilde f_i'\|_{L^\infty(J_i)}+\|\tilde g_i\|_{L^\infty(J_i)}\bigr)^2\to 0,\quad i\to\infty,
		\end{aligned}$$
		a contradiction.
		\smallskip
		
		$\bullet$ Next assume that $a_i\to-\infty$ along a subsequence and that $b_i\ge 0$ is bounded. 
		Owing to \eqref{w-W4}, by passing to a further subsequence, we may assume that there exists $w\in W^{4,\infty}((-\infty,0])$ 
		such that $v_i\to w$ in $W^{3,\infty}_{loc}((-\infty,0])$, hence in particular
		\be{w-normalized}
		w(0)=1.
		\ee
		Fix any $\alpha\in (-\infty,0)$.
		Applying \eqref{intalpha} with $\beta=b_i$ and using the boundary conditions in \eqref{linearpbm} at $\sigma=b_i$ and 
		\eqref{w-W4a}, \eqref{w-W4}, we get, for $i\ge i_0(\alpha)$ large enough,
		$$\begin{aligned}
			\int_\alpha^0 v_i^2+(v_i'')^2
			&\le \int_\alpha^{b_i} v_i^2+(v_i'')^2 
			\le c\bigl|(v_i'v_i''-v_iv_i''')(\alpha)\bigr|+c\e_i^{3/4}\bigl|(\tilde f_iv_i)(b_i)\bigr|
			+c\int_\alpha^{b_i}  (\e_i^{3/4} \tilde f_i'+\tilde g_i)^2 \\
			&\le c\bigl(|v_i(\alpha)|+|v''_i(\alpha)|\bigr)+cM_i^{-1}+cM_i^{-2}(b_i-\alpha).
		\end{aligned}$$
		By Fatou's lemma, it follows that 
		\be{w-Fatou}
		\int_\alpha^0 w^2+(w'')^2\le c\bigl(|w(\alpha)|+|w''(\alpha)|\bigr)\le c
		\quad\hbox{for all $\alpha\in (-\infty,0)$.}
		\ee
		Consequently, $w\in H^2((-\infty,0])$, hence there exists a sequence $\alpha_j\to -\infty$ such that
		$\eta_j:=|w(\alpha_j)|+|w''(\alpha_j)|\to 0$. Going back to \eqref{w-Fatou} with $\alpha=a_j$, we deduce that
		$\int_{\alpha_j}^0 w^2\le c\eta_j$,
		hence $w\equiv 0$ on $(-\infty,0]$ upon letting $j\to\infty$. But this contradicts \eqref{w-normalized}.
		\smallskip
		
		$\bullet$ The cases when $b_i\to\infty$ along a subsequence, with $a_i$ either bounded or unbounded,
		are treated similarly.
		\smallskip
		
		We conclude that \eqref{res1.bis} and \eqref{res1.bis2} are true for $j=0$  and $q=\infty$.
		\smallskip
		
		Now, using $u''''=\e^{-1}(-k u+ g+\e f')$ a.e. in $I$ (cf.~Proposition~\ref{lemresolvent0.bis}(ii)), we deduce that
		$$	\|u''''\|_\infty\le K\e^{-1}\left(\|g\|_\infty +\e^{3/4}\|f\|_\infty+\e\|{f'}\|_\infty\right)$$
		i.e., \eqref{res1.bis}, \eqref{res1.bis2} for $j=4$. The case $j\in\{1,2,3\}$   and $q=\infty$ then follows by interpolation. 
		\smallskip
		
		Finally, the case $q\in(2,\infty)$ follows by interpolating between the cases $q=2$ and $q=\infty$.
		
		\smallskip
		
		(iii)  Let $g\in H^2$ and set $u=A_{2,\e}g$. Since $\e u''''=g-ku \in H^2$, 	we have $u\in H^6$.
			Differentiating twice, we see that $w=u''$ satisfies
			$\e w''''=g''-k''u-2k'u'-ku''$, i.e. 
			$$kw+\e w''''=\hat g:=g''-k''u-2k'u'.$$
			Moreover, $w=w'=0$ on $\partial I$.
			 Multiplying with $w$ and integrating by parts, we obtain
			$$\begin{aligned}
				\int_0^L (k_1 w^2+\e w''^2)ds
				&\le \int_0^L (k w^2+\e w''^2)ds=\bigl[w'w''-w w'''\bigr]_0^L +\int_0^L  \hat gw ds 
				\le \frac{k_1}{2}\int_0^L  w^2ds+\frac{1}{2k_1}\int_0^L \hat g^2 ds.
			\end{aligned}$$
			This combined with $\|u\|_2\le K\|g\|_2$ (cf.~assertion (i)) yields
			$$\|u\|_2^2+\|u''\|_2^2+\e\|u''''\|_2^2\le K\|\hat g\|_2^2\le K\|k\|^2_{H^2}\|g\|^2_{H^2}.$$
			This implies \eqref{res2.bis} for $j\in\{0,2\}$ and the case $j=1$ follows by interpolation.
	\end{proof}

	\section{Local existence-uniqueness}\label{SecLoc}

	For given $\tau\ge 0$, we define our working space
$$X_\tau=L^\infty(-\infty,\tau;H^2(I))$$
with norm
$$\|z\|_{X_\tau}=\sup_{t\in(-\infty,\tau)} \|z(t)\|_{H^2(I)},\quad \tau\ge 0.$$ 
Fix $\e\in(0,1)$ and $\tau>0$.
A (weak) solution of \eqref{mainweak} on $(0,\tau)$ is a function $z\in  X_\tau$
such that
\be{mainweak2}
\left\{\begin{aligned}
	\int_I \cL_\e [z](t) v ds+\int_I \bigl(z''(t) v'' + F'(z'(t)) v'\bigr)ds 
	&= 0,\quad\forall v\in H^2(I), &&a.e.\ t\in(0,\tau),\\
	z(t)&=z_p(t), &&a.e.\  t<0
\end{aligned}
\right.
\ee
(here we omit the variable $s$ and the subscript $\e$ without risk of confusion).

Let the operators $A_{1,\e}, A_{2,\e}$ be defined by 
 Proposition~\ref{lemresolvent0.bis}
with $k(s)=\mu(s)$.
If $z\in   X_\tau$ 
is a solution of \eqref{mainweak2} on $(0,\tau)$,  
then 
 it satisfies
\be{pbmloc2}
z(t)=
\begin{cases}
	A_{1,\e}F'(z'(t))+A_{2,\e}\ds\int_0^\infty z(t-\e a)\rho(\cdot,a) da,&a.e.\ t\in(0,\tau), \\ 
	\noalign{\vskip 1mm}
	z_p(t),& { a.e.\ } t<0.
\end{cases}
\ee
Conversely, for a given function $z\in X_\tau$, we have 
$F'(z'(t))\in L^\infty(I)$ and $\int_0^\infty z(t-\e a)\rho(\cdot,a) da\in L^\infty(I)$ 
for a.e.~$t\in(0,\tau)$, and 
 \eqref{pbmloc2} implies \eqref{mainweak2}.

	The following theorem is our basic local existence-uniqueness result,
	which in turn defines a suitable notion of maximal solution.
	
	\begin{thm} \label{thmloc}
		Assume \eqref{hyppbm1loc}, \eqref{hyppbmloc10} and  $z_p\in X_0$. 
		There exist constants $c_i>0$, with $c_2<1$, 
		depending only on $\rho, L$, such that the following properties are true for all $\e\in(0,\e_1]$,
		{ where}
		\be{BUalt0}
		\e_1:=c_2R_0^{-8}\ \hbox{ and }\  R_0:=1+c_1\|z_p\|_{X_0}.
		\ee
		\begin{itemize}
			\item[(i)] 
			 For $T=c_3\e$, problem \eqref{pbmloc2} admits a (unique) solution 
			$z\in X_T$ such that 
			$$\|z\|_{X_T}\le R_0.$$ 
			\item[(ii)]
			For any $\tau>0$, problem \eqref{pbmloc2} admits at most one solution $z\in X_\tau$ such that
			\be{BUalt2}
			\|z\|_{X_\tau}\le c_4\e^{-1/8}.
			\ee
			\item[(iii)]
			Let $z$ be the local solution given by assertion (i) and set
			$$T^*=\sup\bigl\{\tau>0;\ \hbox{$z$ extends to a solution satisfying \eqref{BUalt2}}\bigr\}.$$
			If $T^*$ is finite, then
			\be{BUalt2b}
			\|z\|_{X_{T^*}}\ge c_5\e^{-1/8}.
			\ee
		\end{itemize}
	\end{thm}

\begin{rem}
	\begin{compactenum}[(i)]
		\item	 Note that, in a standard blow-up alternative, one would normally 
		have $\infty$ in the right hand side of \eqref{BUalt2b}.
		The current definition reflects the lack of parabolicity of the problem
		 (caused by the delay operator, with memory effect of time scale $\varepsilon$,
		which confers to problem \eqref{mainweak2} some elliptic features).
		
		\item The solution $z$ is generally 
		discontinuous at $t=0$, unless a suitable  
		compatibility condition is imposed on $z_p$.
	 An upper estimate of the jump in terms of $\e$ is given in Proposition~\ref{lem.discont}.
	\end{compactenum}
\end{rem}

 In the rest of the paper, by the solution $z=z_\e$ of \eqref{pbmloc}, we will mean the maximal solution
	defined in Theorem~\ref{thmloc}(iii).

Our next result provides, under additional assumptions on the past data,
further time regularity of the solution  for $t>0$, which will be required 
in order to derive the key energy estimates in Section~\ref{SecEn}.

\begin{prop} \label{thmregtime}
	Assume \eqref{hyppbm1loc}, \eqref{hyppbmloc10}, \eqref{hyppbm1locz0} and \eqref{BUalt0}.
	Then 
	\be{zLip0}
	z_\e \in W^{1,\infty}_{loc}( [0,T^*);H^2(I))
	\ee
	and we moreover have
	\be{zLip}
	\|\partial_tz_\e(t)\|_{H^2(I)}\le C(N_t\e^{-1}+R_2) \exp\bigl\{C\e^{-1}t\bigr\}, 
	\quad  \text{a.e.} \ t\in(0,T^*),
	\ee
	with $N_t=\|z_\e\|_{X_{t}}$ and 
	$R_2:=\|z_p\|_{W^{1,\infty}(-\infty,0;H^2(I))}$.
	\end{prop}

The proof of Theorem~\ref{thmloc} will be carried out by a fixed point argument on problem \eqref{pbmloc2}.

\begin{proof}[Proof of Theorem~\ref{thmloc}]
	 We first note that, owing to  \eqref{hyppbm1loc}, \eqref{hyppbmloc10}, assumptions \eqref{hyplinearpbm}, \eqref{hypK1} of 
		Proposition~\ref{lemresolvent.bis} with $k(s)=\mu(s)$ are satisfied. 
		
		\smallskip
		
		(i) For $T>0$, $R\ge  1+\|z_p\|_{X_0}$ to be determined below, we set
		$$B_{T,R}=\left\{z\in X_T,\ \|z\|_{X_T}\le R \hbox{ and $z(t)=z_p$ for a.e.~$t<0$}\right\},$$
		and consider the fixed point operator
		$\mathcal{T}:B_{T,R}\to X_T$
		defined by
		\be{pbmloc3}
		[\mathcal{T}(z)](t)=
		\begin{cases}
			A_{1,\e}F'(z'(t))+A_{2,\e}\ds\int_0^\infty z(t-\e a)\rho(\cdot,a) da,&a.e.\ t\in(0,T), \\ 
			\noalign{\vskip 1mm}
			z_p(t),& a.e.~t<0.
		\end{cases}
		\ee
		Note that $B_{T,R}$, being a closed subset of the Banach space $X_T$,
		is a complete metric space.
		Recalling \eqref{hyppbm1loc}, we set
		\be{defM1}
		M_1=1+\|\rho\|_{L_a^\infty(0,\infty;H^2(I))}+\|\rho\|_{L_a^1(0,\infty;H^2(I))}.
		\ee
		In what follows we shall respectively denote by $\|\cdot\|_\infty$ and $\|\cdot\|_{H^2}$
		the norms in $L^\infty(I)$ and $H^2(I)$, and keep the variable $s\in I$ implicit.
		Let $z\in B_{T,R}$. 
		Using $F'(z')=4z'(1-{z'}^2)$, 
		estimates \eqref{EstimH2}, \eqref{res2.bis},
		the Sobolev inequality $\|v\|_{W^{1,\infty}(I)}\le C\|v\|_{H^2(I)}$
		and the fact that $\|fg\|_{H^2}\le C\|f\|_{H^2}\|g\|_{H^2}$, we obtain,  for a.e.\ $t\in(0,T)$,
		$$\begin{aligned}
			\|\mathcal{T}(z)(t)\|_{H^2} 
			&\le C\|A_{1,\e}F'(z'(t))\|_{H^2}+\Big\|A_{2,\e}\int_0^{t/\e} z(t-\e a)\rho(a) da \Bigr\|_{H^2} 		+\Big\|A_{2,\e}\int_{t/\e}^\infty z_p(t-\e a)\rho(a) da \Bigr\|_{H^2}\\
			&\le C\e^{1/4} \|F'(z'(t))\|_{L^2}  
			+C\Big\|\int_0^{t/\e} z(t-\e a)\rho(a) da \Bigr\|_{H^2} 
			+C\Big\|\int_{t/\e}^\infty z_p(t-\e a)\rho(a) da \Bigr\|_{H^2}\\
			&\le C\e^{1/4} (1+\|z'(t)\|^2_\infty) \|z'(t)\|_\infty 
			+CM_1\e^{-1}T\sup_{\sigma\in(0,t)}\|z(\sigma)\|_{H^2}
			+CM_1\sup_{\sigma<0}\|z_p(\sigma)\|_{H^2}.
		\end{aligned}$$
		It follows that
		\be{mapping1}
		\|\mathcal{T}(z)\|_{X_T}\le C_0\bigl(\e^{1/4}R^3+M_1\e^{-1}TR+M_1 \|z_p\|_{X_{0}}\bigr),
		\ee
		where $C_0=C_0(\rho,L)\ge 1$ 
			(note that, for a.e.\ $t<0$, we have $\|z(t)\|_{H^2}=\|z_p(t)\|_{H^2}\le C_0M_1 \|z_p\|_{X_{0}}$ since 
			$C_0, M_1\ge 1$).
		Similarly, for all $z_1, z_2\in B_{T,R}$, using also
		\be{zcube}
		|X^3-Y^3|\le 2(X^2+Y^2)|X-Y|,\quad X,Y\in\R,
		\ee
		we get
		$$\begin{aligned}
			&\|\mathcal{T}(z_1)(t)-\mathcal{T}(z_2)(t)\|_{H^2}\\
			&\le C\bigl\|A_{1,\e}\bigl(F'(z'_1(t))-F'(z'_2(t))\bigr)\bigr\|_{H^2}
			+\Big\|A_{2,\e}\int_0^{t/\e} [z_1(t-\e a)-z_2(t-\e a)]\rho(a) da\Bigr\|_{H^2}\\
			&\le C\e^{1/4} \bigl\|F'(z'_1(t))-F'(z'_2(t))\bigr\|_{L^2}
			+\Big\|A_{2,\e}\int_0^{t/\e} [z_1(t-\e a)-z_2(t-\e a)]\rho(a) da\Bigr\|_{H^2}\\
			&\le C\e^{1/4} \bigl(1+\|z'_1(t)\|_\infty+\|z'_2(t)\|_\infty\bigr)^2\|(z'_1-z'_2)(t)\|_\infty +C\Big\|\int_0^{t/\e} [z_1(t-\e a)-z_2(t-\e a)]\rho(a) da\Bigr\|_{H^2}\\
			&\le C\e^{1/4} R^2\|(z'_1-z'_2)(t)\|_\infty
			+CM_1\e^{-1}T\sup_{\sigma\in(0,t)}\|(z_1-z_2)(\sigma)\|_{H^2}.
		\end{aligned}$$
		It follows  (taking $C_0=C_0(\rho,L)$ larger if necessary) that
		\be{mapping2}
		\|\mathcal{T}(z_1)-\mathcal{T}(z_2)\|_{X_T}\le C_0\bigl(\e^{1/4}R^2+M_1\e^{-1}T\bigr)\|z_1-z_2\|_{X_T}.
		\ee
		
		Now set 
		\be{choiceci}
		c_1:=2C_0M_1\quad
		c_2:=4^{-8}C_0^{-4}, \quad
		c_3=(4C_0M_1)^{-1},\quad 
		c_4:=\ts(4C_0)^{-1/2}
		\ee
		and
		$$R_0:=1+c_1\|z_p\|_{X_0}.$$
		For $\e\in(0,1]$, choose $T:=c_3\e$ and any $R$ such that
		\be{choiceR}
		R_0\le R\le c_4\e^{-1/8},
		\ee
		we have
		\be{mapping3}
		C_0\bigl\{\e^{1/4}R^2+M_1\e^{-1}T\bigr\}\le 1/2
		\ee
		and 
		\be{mapping4}
		C_0\bigl(\e^{1/4}R^3+M_1\e^{-1}TR+M_1\|z_p\|_{X_0}\bigr)-R
		\le \bigl\{C_0\bigl(\e^{1/4}R^2+M_1\e^{-1}T\bigr)-\ts\frac12\bigr\}R\le 0.
		\ee
		It follows from \eqref{mapping3}-\eqref{mapping4}
		that 
		\be{mapping0}
		\hbox{$\mathcal{T}$ is a contraction mapping on $B_{T,R}$.}
		\ee
		Consequently, \eqref{pbmloc2} admits a unique fixed point in $B_{T,R}$.
		Assuming $\e\in(0,\e_1]$ with $\e_1:=c_2R_0^{-8}$ and making the particular choice $R=R_0$ in \eqref{choiceR},
		this proves assertion (i).
		
		\smallskip
		
		(ii) Let $\tau>0$ and let $z_1,z_2\in X_\tau$ be solutions satisfying \eqref{BUalt2}.
		Let 
		$$\tau_0=\sup\,\Bigl\{t\in [0,\tau);\ z_1=z_2\ \hbox{a.e. on $(-\infty,t)$}\Bigr\}.$$
		By \eqref{mapping0}
		we know that $\tau_0>0$. Assume for contradiction that
		$\tau_0<\tau$.
		Since $\|z_i\|_{X_\tau}\le R:=c_4\e^{-1/8}$ 
		and $z_1=z_2$ for $t<\tau_0$, the argument leading to \eqref{mapping2},
			with $t/\e$ replaced by $(t-\tau_0)/\e$ implies
		$$\|z_1-z_2\|_{X_t} \le C_0\bigl\{R^2\e^{1/4}+M_1\e^{-1}(t-\tau_0)\bigr\}\|z_1-z_2\|_{X_t},
		\quad \tau_0<t<\tau.$$
		Since $C_0R^2\e^{1/4} \le\ 1/4$ (cf.~\eqref{choiceci}  and \eqref{choiceR}), 
		we deduce that $z_1(t)=z_2(t)$ for $t\ge \tau_0$ close to $\tau_0$: a contradiction.
		
		\smallskip
		
		(iii) By the definition of $T^*$, for each integer $j>{T^*}^{-1}$, there exists a solution 
		$z_j \in X_{T^*-j^{-1}}$ satisfying \eqref{BUalt2} with $\tau=T^*-j^{-1}$.
		By the uniqueness statement in assertion~(ii), we have $z_j=z_{j+1}$ for $t\le T^*-j^{-1}$.
		The desired solution is thus obtained by setting $z:=z_j$ for $t<T^*-j^{-1}$ and $j>{T^*}^{-1}$.

		Next assume for contradiction that $T^*<\infty$ and $\|z\|_{X_{T^*}}\le c_5\e^{-1/8}$,
		where $c_5=c_4/2c_1$.
		Since $\e\le\e_1\le c_2= (c_4/2)^8$, we have
		$\e^{1/8}\le c_4/2=c_1c_5$. 
		It follows that
		$$R_1:=1+c_1\|z\|_{X_{T^*}}\le 1+c_1c_5\e^{-1/8}\le 2c_1c_5\e^{-1/8}=c_4\e^{-1/8}.$$
		 Set $\hat z_p(t)=z(t+T^*)$ for a.e.~$t<0$. We may thus
		apply the above fixed point argument with $\tilde z_p\in X_0$ instead of $z_p$ and
		$R=R_1$ instead of $R=R_0$ in \eqref{choiceR}.
		This provides a time $T>0$ and a function $\hat z\in X_T$ such that
			$\hat z=\hat z_p$ for a.e. $t<0$ and
			$$\hat z(t)=A_{1,\e}F'(\hat z'(t))+A_{2,\e}\ds\int_0^\infty \hat z(t-\e a)\rho(s,a) da,\quad a.e.\ t\in(0,T).$$
			Letting $\tilde z(t):=\hat z(t-T^*)$ for a.e.~$t<T^*+T$, it follows that $\tilde z$ satisfies
			$\tilde z(t)=z(t)$ for a.e.~$t<T^*$ and
			$$\tilde z(t)=A_{1,\e}F'(\tilde z'(t))+A_{2,\e}\ds\int_0^\infty \tilde z(t-\e a)\rho(s,a) da,\quad a.e.\ t\in(0,T^*+T),$$
			along with \eqref{BUalt2}: a contradiction.
\end{proof}

\begin{proof}[Proof of Proposition~\ref{thmregtime}]
We must show that, more precisely, there exists a representative $\tilde z$ of $z$ (i.e.~$\tilde z(t)=z(t)$~a.e.)
such that, for each $T\in (0,T^*)$, $\tilde z$ is Lipschitz continuous from $[0,T]$ to $H^2(I)$ and satisfies the estimate in \eqref{zLip}.
Fixing a representative $z$ and $T\in(0,T^*)$, it suffices to show that there exists $\Sigma\subset (0,T)$
with $|(0,T)\setminus \Sigma|=0$, such that
	\be{LipAux0}
	\|z(t_2)-z(t_1)\|_{H^2}\le C(N_T\e^{-1}+R_2)|t_2-t_1| \exp\bigl\{C\e^{-1}T\bigr\}, \quad\text{ for all  }t_1,\ t_2\in \Sigma
	\ee
	(since one then easily sees that $\tilde z(t):=\lim_{\Sigma\ni t'\to t} z(t')$ exists in $H^2$ for all $t\in [0,T]$,
	coincides with $z$ on $\Sigma$, and has the desired properties).

To this end, we first note that, by \eqref{pbmloc2}, there exists $\Sigma \subset (0,T)$
with $|(0,T)\setminus \Sigma|=0$, such that
	\be{LipAux1}
	z(t)=A_{1,\e}F'(z'(t))+A_{2,\e}\ds\int_0^\infty z(t-\e a)\rho({s,a}) da \quad\text{ for all }t\in J.
	\ee
Fix $t_1,\ t_2\in \Sigma$ with $t_1<t_2$ and set $h:=t_2-t_1$ and $N:=\|z\|_{X_T}$.
	Using \eqref{LipAux1}, \eqref{EstimH2}, \eqref{res2.bis},
	\eqref{zcube} with $X=z'(s,t)$, $Y=z'(s,t-h)$,
	the Sobolev inequality $\|v\|_{W^{1,\infty}(I)}\le C\|v\|_{H^2(I)}$
	and the fact that $\|fg\|_{H^2}\le C\|f\|_{H^2}\|g\|_{H^2}$,
	we deduce that, for all $t\in \Sigma\cap (h+J)$,
	$$\begin{aligned}
		\bigl\|z(t)-z(t-h)\bigr\|_{H^2}
		&\le 4\Bigl\|A_{1,\e}\Bigl\{[z'(1-{z'}^2)](t)-[z'(1-{z'}^2)](t-h)\Bigr\}\Bigr\|_{H^2}+\Big\|A_{2,\e}\int_0^\infty [z(t-\e a)-z(t-h-\e a)]\rho(a) da\Bigr\|_{H^2}\\
		&\le C\e^{1/4} (1+N^2) \|z'(t)-z'(t-h)\|_2+C\int_0^\infty \Big\|[z(t-\e a)-z(t-h-\e a)]\rho(a) \Bigr\|_{H^2} da\\
		&\le C_1\e^{1/4} (1+N^2)  \|z(t)-z(t-h)\|_{H^2}+C_1\int_0^\infty \Big\|[z(t-\e a)-z(t-h-\e a)] \Bigr\|_{H^2} \|\rho(a)\|_{H^2} da,
	\end{aligned}$$
	with $C_1=C_1(\rho, L)>0$.
	On the other hand, setting $\delta_h(t)=\|z(t)-z(t-h)\|_{H^2}$, we note that, for each $t\in (0,T)$, we have
			\be{deltah0}
			\delta_h(t-\e a)=\|z_p(t-\e a)-z_p(t-\e a-h)\|_{H^2}\le R_2h,\quad \hbox{ for a.e.~}a>t/\e.
			\ee
	Moreover, up to replacing $C_0$ in 
	\eqref{mapping1}, \eqref{mapping2} by $\max(C_0,$ $C_1)$, we may assume that
	$$C_1\e^{1/4} (1+N^2)  \le C_1\e_1^{1/4} + C_1c_4^2
	=4^{-2}C_1C_0^{-1}R_0^{-2}+C_1(4C_0)^{-1}\le \ts\frac12.$$
	Using \eqref{defM1} and \eqref{deltah0}, we then obtain, for all $t\in \Sigma \cap (h+\Sigma )$,
	$$\begin{aligned}
		(2C_1)^{-1}\delta_h(t)
		&\le \int_0^\infty \Big\|[z(t-\e a)-z(t-h-\e a)] \Bigr\|_{H^2} \|\rho(a)\|_{H^2} da\\
		&=\int_0^{(t-h)/\e} \delta_h(t-\e a) \|\rho(a)\|_{H^2} da  +\int_{(t-h)/\e}^{t/\e} \delta_h(t-\e a) \|\rho(a)\|_{H^2} da 
		+\int_{t/\e}^\infty \delta_h(t-\e a) \|\rho(a)\|_{H^2} da \\
		&=\e^{-1}\int_h^t \delta_h(\tau) \|\rho(\e^{-1}(t-\tau))\|_{H^2} d\tau +\int_{(t-h)/\e}^{t/\e} \delta_h(t-\e a) \|\rho(a)\|_{H^2} da 
		+\int_{t/\e}^\infty \delta_h(t-\e a) \|\rho(a)\|_{H^2} da \\
		&\le M_1\left(\e^{-1}\int^t_h\delta_h(\tau) d\tau+2N\e^{-1}h+R_2h\right),
	\end{aligned}$$ 
	hence
		\be{LipAux2}
		\delta_h(t)\le C(N\e^{-1}+R_2)h+C\e^{-1}\int^t_h \delta_h(\tau) d\tau =:G(t),\quad \hbox{for all $t\in J\cap (h+J)$.}
		\ee
	Now observing that $|(h,T)\setminus(\Sigma \cap (h+J))|=0$, it follows that the (absolutely continuous) function $G$ satisfies 
	$\bigl(\exp\bigl\{-C\e^{-1}t\bigr\}G(t)\bigr)'\le 0$ for a.e.~$t\in(h,T)$, hence $G(t)\le G(h)\exp\bigl\{C\e^{-1}(t-h)\bigr\}$
	for all $t\in(h,T)$. Going back to \eqref{LipAux2} and noting that $t_2\in \Sigma \cap (h+\Sigma)$ owing to $t_1,\ t_2\in \Sigma$, we obtain
	$$\|z(t_2)-z(t_1)\|_{H^2}=\delta_h(t_2)\le G(t_2)\le G(h)\exp\bigl\{C\e^{-1}T\bigr\}=C(N\e^{-1}+R_2)|t_2-t_1| \exp\bigl\{C\e^{-1}T\bigr\},$$
	hence \eqref{LipAux0}.
	The proposition follows.
\end{proof}

\section{Energy estimates}\label{SecEn}

Let $z=z_\e$ be the solution given by Theorem~\ref{thmloc}, of existence time $T^*_\e$, and denote
\be{defusta}
u_\e (s,t,a) :=  z_\e(s,t)-z_\e(s,t-\e a),\quad 0<t<T^*_\e,\ a>0.
\ee
We define the energy:
\be{defEeps}
E_\e(t) :=\frac{1}{2\e}\int_I \intrp u_\e^2(s,a,t)  \rho(s,a) da ds + \frac{1}{2}\int_I (z_\e''(t))^2 ds+ 
\int_I F(z_\e'(t)) ds. 
\ee
In the following two propositions, we obtain the monotonicity 
 and dissipation property of the energy and,
as a consequence, derive some key estimates.
Special care is given to the dependence with respect to $\e$, 
which will turn out important in the proof of our main results.

\begin{prop}\label{thm.nrj}
	Assume \eqref{hyppbm1loc},  \eqref{eq.kernel0}, \eqref{hyppbmloc10}, \eqref{hyppbm1locz0}.
	 Let $\e\in(0,\e_1)$ where $\e_1$ is given by Theorem~\ref{thmloc}.
	
	(i) We have $E_\e\in W^{1,\infty}_{loc}([0,T^*_\e))$ and
	\be{eq.dissipation0}
	E_\e'(t) = {\frac{1}{2\e^2}}\int_I \intrp u_\e^2 (s,a,t) \partial_a\rho(s,a) da ds\le 0,\quad a.e.\ t\in(0,T^*_{\e}).
	\ee
	
	(ii) For all $0<t_0<t<T^*_{\e}$, we have
	\be{eq.dissipation}
	\int_{t_0}^t\int_I \intrp  u_\e^2 (s,a,\tau) |\partial_a\rho(s,a)| da ds d\tau
	\le 2 \e^2 E_\e(t_0),
	\ee
	\be{globalboundz0step}
	\int_{t_0}^t \bigl\{\|\partial_t z_\e(\tau)\|_{L^2(I)}^2 + \e \|\partial_t z_\e''(\tau)\|_{L^2(I)}^2\bigr\} d\tau \le
	CE_\e(t_0),
	\ee
	\be{globalboundz1step}
	\int_{t_0}^t \big\{ \e^{1/4}\|\partial_t z_\e(\tau)\|_{L^\infty(I)}^2
	+\e^{3/4} \|\partial_t z_\e'(\tau)\|_{L^\infty(I)}^2\bigr\} d\tau \le CE_\e(t_0).
	\ee
\end{prop}

 In the next proposition we make the following assumption (which is a consequence of the second part of \eqref{hyp10}):
\be{condrhoa}
\int_0^\infty a\|\rho(\cdot,a)\|_2 da <\infty.
\ee

\begin{prop}\label{thm.dt.z.unif}
	Assume \eqref{hyppbm1loc},  \eqref{eq.kernel0}, \eqref{hyppbmloc10}, \eqref{hyppbm1locz0}, \eqref{condrhoa}, let 
	$$R:=1+\|z_p\|_{X_{0}},\quad \hat R:=\|z_p\|_{W^{1,\infty}(-\infty,0;L^\infty(I))},
		\quad \bar R:=R^4+\hat R^2$$
	and  let $\e\in(0,\e_1)$ where $\e$ is given by Theorem~\ref{thmloc}.
		Then
	\be{Energy-init2}
	E_\e(t)\le C\bar R,\quad 0<t<T_\e^*,
	\ee
	and
	\be{globalboundz1}
	\|z_\e'(\cdot,t)\|_{L^\infty(I)}\le C\bar R^{1/3},\quad \|z_\e''(\cdot,t)\|_{L^2(I)}\le C\bar R^{1/2},\quad 0<t<T_\e^*.
	\ee
\end{prop}

	In the process of proving Proposition~\ref{thm.dt.z.unif}, we obtain an estimate for the discontinuity of $z_\e$ at $t=0$,
	of independent interest.
	
	\begin{prop}\label{lem.discont}
		Under the assumptions of Proposition~\ref{thm.dt.z.unif}, there exists
		\be{zeps-init}
		z_\e(0^+):= \lim_{t\to 0^+} z_\e(t) \quad\hbox{ in the {strong} $H^2(I)$ sense}
		\ee
		and we have
		\be{zeps-discont}
		\|z_\e(0^+)-z_p(0)\|_2\le C(R+\hat R)\e^{1/2}.
		\ee
	\end{prop}

	In view of the proofs of these three propositions, we prepare the following two lemmas.

\begin{lem}\label{lem.dt.z.unif}
	Let $\ell$ be a positive integer, 
	$v\in H^2(I)$ and set
	\be{Gelldef}
	G(t) := \int_I \intrp u^\ell_\e(s,t,a)  \rho(s,a) v(s)da ds.
	\ee
	We have $G\in W^{1,\infty}_{loc}([0,T_\e^*))$ and, for a.e.~$t\in(0,T_\e^*)$,
	\be{Gellprime}
	G'(t)=\ell\ds\int_I\int_0^\infty \partial_tz_\e(s,t) u^{\ell-1}_\e(s,t,a)\rho(s,a)v(s) da ds
	+ \frac1\e\ds\int_I\int_0^\infty u^\ell_\e(s,t,a)\partial_a\rho(s,a)v(s) da ds.
	\ee
\end{lem}

\begin{lem}\label{lem.dt.z.unif2}
	For a.e.~$t\in(0,T_\e^*)$ and all $v\in H^2(I)$, the time derivative $\dt z$ satisfies
	\be{eq.dt.z}
	\int_I \mu(s) \dt z_\e(s,t) v(s)ds + \e \int_I \bigl( \dt z''_\e v'' +F''(z'_\e) \dt z'_\e v'\bigr)(s,t)ds
	+\frac1\e\int_I \intrp  u_\e(s,t,a) v(s) \partial_a \rho(s,a) da ds =0. 
	\ee	
\end{lem}


	\begin{rem}
		Setting $w(s):=\dt z(s,t)$ and assuming that the quantity $\Xi := \frac1\e\intrp  u_\e(s,t,a) \partial_a \rho(s,a) da $ is known,  Lemma~\ref{lem.dt.z.unif2} means that $w$  is a weak solution of
		the (closed) elliptic equation
		$$ \mu w+ \e  \bigl( w''''-(F''(z'_\e) w')'\bigr)
		=-\Xi ,\quad s\in I,$$
		along with the natural boundary conditions.
	\end{rem}

\begin{proof}[Proof of Lemma~\ref{lem.dt.z.unif}]
	Setting $\tilde u(s,t,b)=z(s,t)-z(s,b)$ and changing variables by $b=t-\e a$, we first rewrite
	$G(t)= \e^{-1}\int_I \int_{-\infty}^t  \tilde u^\ell(s,t,b)\rho(s,\ts\frac{t-b}{\e}) v(s)db ds$.
	It follows that,  for all $t,h$ with $t,t+h\in(0,T^*_\e)$,
	$$\begin{aligned}
		& \e(G(t+h)-G(t))=\ds\int_I\int_{-\infty}^t \bigl(\tilde u^\ell(s,t+h,b)-\tilde u^\ell(s,t,b)\bigr)\rho(s,\ts\frac{t+h-b}{\e})v db ds\\
		&+\ds\int_I\int_{-\infty}^t \tilde u^\ell(s,t,b)\bigl(\rho(s,\ts\frac{t+h-b}{\e})-\rho(s,\ts\frac{t-b}{\e})\bigr) v db ds 
		+ \ds\int_I\int_0^h \tilde u^\ell(s,t+h,t+\tau)\rho(s,\ts\frac{h-\tau}{\e}) v d\tau ds.
	\end{aligned}$$
	Using assumption \eqref{eq.kernel0} 
	and the fact that $z\in L^\infty((-\infty,T_0)\times I)\cap W^{1,\infty}(0,T_0;L^\infty(I))$
	for any $T_0<T^*$ (cf.~Proposition~\ref{thmregtime}), we deduce that $G\in W^{1,\infty}_{loc}([0,T^*))$ and,
	taking into account that $\tilde u(\cdot,t,t)=0$, 
	we obtain
	$$G'(t)
	= \e^{-1}\ds\int_I\int_{-\infty}^t \partial_t\tilde u^\ell(s,t,b)\rho(s,\ts\frac{t-b}{\e})v db ds
	+\e^{-2}\ds\int_I\int_{-\infty}^t \tilde u^\ell(s,t,b)\partial_a\rho(s,\ts\frac{t-b}{\e})v db ds,$$
	for a.e.~$t\in (0,T^*)$. The result follows by going back to the variable $a$.
\end{proof}

\begin{proof}[Proof of Lemma~\ref{lem.dt.z.unif2}]
	Fix $v\in H^2(I)$ and let $G$ be given by \eqref{Gelldef} with $\ell=1$.
	 Since $z\in W^{1,\infty}_{loc}([0,T^*);$ $H^2(I))$ by Proposition~\ref{thmregtime}, we may differentiate \eqref{mainweak} in time which, after using \eqref{Gellprime}, gives
	$$\begin{aligned}
		-\e & \ddt{} \int_I\ \bigl(z''v'' + F'(z') v'\bigr)ds =G'(t)=\int_I\intrp \dt z\rho  vdads 
		+\e^{-1}\int_I \intrp  u \partial_a \rho v da ds,
	\end{aligned}$$
	hence \eqref{eq.dt.z}.
\end{proof}

\begin{proof}[Proof of Proposition~\ref{thm.nrj}]
	
		(i) Set $E_1(t)= \frac12 \int_I |z''|^2 ds +  \int_I F(z')ds$. 
		Since $z\in W^{1,\infty}_{loc}([0,T^*);$ $H^2(I))$, we have
		$E_1\in W^{1,\infty}_{loc}([0,T^*))$ and,
		using $z_t(\cdot,t)$ as test-function in the weak formulation \eqref{mainweak}, we get, for a.e.~$t\in(0,T^*)$,
		\be{E1prime}
		E'_1(t)
		= \int_I z''\dt z'' ds + \int_i F'(z')\dt z'ds=  -\frac{1}{\e}\int_I\intrp u\dt z \rho dads.
		\ee
		Let now $G$ be given by \eqref{Gelldef} with $\ell=2$ and $v\equiv 1$.
		Since $E_\e(t)=E_1(t)+\frac{1}{2\e}G(t)$, we obtain
		\eqref{eq.dissipation0} by combining \eqref{Gellprime} and \eqref{E1prime}.
		\smallskip
		
	 (ii) Property \eqref{eq.dissipation} readily follows from \eqref{eq.dissipation0}.
	Set $$\cS(s,t) := \intrp u^2(s,t,a)|\partial_a\rho(s,a)|da.$$ Note that, by Cauchy-Schwarz and \eqref{eq.kernel0},
	$$
	\Bigl(\intrp u(s,t,a)|\partial_a\rho(s,a)| da\Bigr)^2\le \cS(s,t)\intrp |\partial_a\rho| da \le C\cS(s,t).$$
	For any $\lambda>0$, applying \eqref{eq.dt.z}  with $v=\dt z(\cdot,t)$ and using the fact that $F''\ge -4$, we deduce that,
	for  a.e.~$t\in(0,T^*)$,
	$$\begin{aligned}
		\int_I \bigl(\mu |\dt z|^2 + \e |\dt z''|^2 \bigr)ds
		&=-\e \int_I F''(z') |\dt z'|^2-\e^{-1}\int_I \dt z\left(\intrp  u  \partial_a \rho da\right) ds\\
		&\le 	4\e \int_I  |\dt z'|^2ds+\lambda\int_I |\dt z|^2ds+C\lambda^{-1}\e^{-2}\int_I \cS(s,t) ds.
	\end{aligned}$$
	Next using
	$\int_I  |\dt z'|^2 \le \frac{1}{8}\int_I  |\dt z''|^2+ C_2\int_I  |\dt z|^2$
	 with $C_2=C_2(L)>0$ (cf.~\eqref{interpol-eta}), it follows that
	$$\begin{aligned}
		(\mu_{min}-\lambda-4C_2\e)\int_I |\dt z|^2 ds + \frac{\e}{2} \int_I |\dt z''|^2 ds
		&\le 	C\lambda^{-1}\e^{-2}\int_I  \intrp u^2  |\partial_a\rho| da ds.
	\end{aligned}$$
	Up to replacing $C_0$ in 
	\eqref{mapping1}, \eqref{mapping2} by $\max(C_0,(C_2/\mu_{min})^{1/4})$, we may assume that
	$\e\le\e_1\le 4^{-8}C_0^{-4}\le\mu_{min}/(8C_2)$.
	Choosing $\lambda=\mu_{min}/4$, we get
	\be{choicelambda}
	A(t):=\int_I |\dt z|^2 ds +\e \int_I |\dt z''|^2 ds\le C\e^{-2}\int_I  \intrp u^2  |\partial_a\rho| da ds.
	\ee
	 Integrating this in time and using \eqref{eq.dissipation} yields \eqref{globalboundz0step}.
		
		Let us finally check \eqref{globalboundz1step}.
	By \eqref{interpol-eta} with $\eta=\e^{1/2}$, we have
	\be{dtzprimeL2}
	\|\partial_t z'\|^2_{L^2}\le \e^{1/2}\|\partial_t z''\|^2_{L^2}+C\e^{-1/2}\|\partial_t z\|^2_{L^2}
	\le C\e^{-1/2} A(t).
	\ee
	Using the Sobolev inequality
	\be{Sobolev-eta}
	\|v\|_{L^\infty}\le \eta\|v'\|_{L^2}+ C(1+\eta^{-1})\|v\|_{L^2}, \quad v\in H^1(I),
	\ee
	with $\eta=\e^{1/4}$, we deduce from \eqref{dtzprimeL2} that
	\be{dtzprimeLinfty}
	\|\partial_t z'\|^2_{L^\infty}
	\le  \eta\|\partial_t z''\|^2_{L^2}+C \eta^{-1}\e^{-1/2} A(t)
	\le C\e^{-3/4} A(t)
	\ee
	and
	$$\|\partial_t z\|^2_{L^\infty}\le \eta\|\partial_t z'\|^2_{L^2}+ C\eta^{-1}\|\partial_t z\|^2_{L^2}
	\le C \bigl( \eta\e^{-1/2} A(t)+ \eta^{-1}\|\partial_t z\|^2_{L^2}\bigr)
	\le C  \e^{-1/4} A(t),$$
	 and \eqref{globalboundz1step} follows from  \eqref{globalboundz0step}.
\end{proof}

We next prove Proposition~\ref{lem.discont} as a consequence of Proposition~\ref{thm.nrj}.

\begin{proof}[Proof of Proposition~\ref{lem.discont}]
	 Property \eqref{zeps-init} follows from $z\in W^{1,\infty}_{loc}([0,T^*);H^2(I))$ 
		(cf.~Proposition~\ref{thmregtime}).
	
	\smallskip
	
	We next prove \eqref{zeps-discont}.
	By Theorem~\ref{thmloc}(i), since $0<\e\le \e_1=c_2R^{-8}$,
	we have
	$T^*\ge T_\e:= c_3\e$
	and
	\be{H2R}
	\|z_\e(t)\|_{H^2}\le  CR,\quad -\infty<t\le T_\e.
	\ee
	Let now $t\in(0,T_\e)$. For all $s\in I$, we write
	\be{H2R1}
	\mu(s)(z(s,t)-z_p(s,0))=\e \bigl(-z''''+\partial_s(F'(z'))\bigr)+\int_0^\infty (z(s,t-\e a)-z_p(s,0))\rho(s,a) da.
	\ee
	To estimate the first term on the right hand side, we use \eqref{pbmloc}, \eqref{res1.bis} with $j=4$, \eqref{res2.bis} with $j=2$, 
	the fact that $\|fg\|_{H^2}\le C\|f\|_{H^2}\|g\|_{H^2}$, and \eqref{H2R} to get
	$$\begin{aligned}
		\|z(t)\|_{H^4} 
		&\le \|A_{1,\e}F'(z'(t))\|_{H^4}+\Big\|A_{2,\e}\int_0^\infty z(t-\e a)\rho(a) da \Bigr\|_{H^4}\\
		&\le C\e^{-1/4} \|F'(z'(t))\|_{L^2}  +C\|F''(z'(t))z''(t)\|_{L^2} +C\e^{-1/2}\Big\|\int_0^\infty z(t-\e a)\rho(a) da \Bigr\|_{H^2}\\
		&\le C(1+\|z'(t)\|^2_\infty)(\e^{-1/4}  \|z'(t)\|_\infty +\|z''(t)\|_2)
		+CM_1\e^{-1/2}\sup_{\sigma<t}\|z(\sigma)\|_{H^2}\\
		&\le CR^3\e^{-1/4}+CM_1R\e^{-1/2}\le CR\e^{-1/2},
	\end{aligned}$$
	hence
	\be{H2R2}
	\bigl\|-z''''(t)+(F'(z'(t))'\bigr\|_2
	\le CR\e^{-1/2}+CR^3\le CR\e^{-1/2}.
	\ee
	On the other hand,  by dominated convergence, \eqref{hyppbm1locz0} and  \eqref{condrhoa}, we have
	$$\begin{aligned}
		\lim_{t\to 0^+}\int_0^\infty \|z(t-\e a)-z_p(0)\|_\infty\|\rho(\cdot,a)\|_2 da 
		&=\int_0^\infty \|z_p(-\e a)-z_p(0)\|_\infty \|\rho(\cdot,a)\|_2 da\le \hat R\e \int_0^\infty a\|\rho(\cdot,a)\|_2 da=C\hat R\e.
	\end{aligned}$$
	Combining this with \eqref{H2R1} and \eqref{H2R2}, it follows that
	$$\mu_{min}\limsup_{t\to 0^+}\|(z(t)-z_p(0)\|_2\le CR\e^{1/2}+C\hat R\e, 
	$$
	hence \eqref{zeps-discont}.
\end{proof}

We finally prove Proposition~\ref{thm.dt.z.unif} as a consequence of Proposition~\ref{lem.discont}.

\begin{proof}[Proof of Proposition~\ref{thm.dt.z.unif}]
	By dominated convergence, \eqref{hyppbm1locz0} and \eqref{zeps-init}, we have
	$$
	\lim_{t\to 0^+} \int_I \intrp u_\e^2(s,a,t)  \rho(s,a) da ds
	=\int_I \intrp |z(s,0^+)-z_p(s,-\e a)|^2  \rho(s,a) da ds.$$
	Moreover, for each $a>0$, we have
	$$\begin{aligned}
		\int_I & |z(s,0^+)-z_p(s,-\e a)|^2  ds \le 2\int_I |z(s,0^+)-z_p(s,0)|^2  ds+2\int_I |z_p(s,0)-z_p(s,-\e a)|^2ds \\
		&\le C(R+\hat R)^2
		\e+ C\hat R^2\min(1,\e^2a^2)  
		\le C(R+\hat R)^2\e(1+a).
	\end{aligned}$$
	Therefore
	$$
\begin{aligned}
		\lim_{t\to 0^+} \int_I \intrp u_\e^2(s,a,t)  \rho(s,a) da ds & 
	\le C(R+\hat R)^2\e
	\int_I \intrp (1+a)\rho(s,a) da ds \lesssim C(R+\hat R)^2\e.
\end{aligned}
	$$ 
	This, along with \eqref{H2R} and the definition of $E_\e$ guarantees that $\lim_{t\to 0^+} E_\e(t)\le C\bar R$
	and \eqref{Energy-init2} follows from the monotonicity \eqref{eq.dissipation0}.
	Property \eqref{globalboundz1} then follows from \eqref{Energy-init2} using the inequality 
	$\|v\|_\infty^3\le C(\|v^3\|_1+\|(v^3)'\|_1)\le C(1+\|v\|_4^4+\|v'\|_2^2)$.
\end{proof}

\section{Uniform $L^\infty$ bounds and global existence (proof of Theorem~\ref{thmglob}{\rm (i)})}  \label{SecUnif}

To prove global existence we cannot solely rely on 
the energy estimates from the previous section, since  they give 
bounds up to $T^*$ for $z'_\e$ (and $z''_\e$) but not for $z_\e$ itself.
The latter are provided by the following proposition,
which immediately implies Theorem~\ref{thmglob}(i).

\begin{prop} \label{thmunifbounds}
	Assume \eqref{hyppbm1loc}--\eqref{hyp10}, \eqref{hyppbm1locz0},
	and let $\e_1, \bar R$ be respectively given by Theorem~\ref{thmloc} and Proposition~\ref{thm.dt.z.unif}.
	There exists $c_0>0$ depending only on $\rho, L$ such that,
	for each $\e\in(0,\e_0]$ with $\e_0=c_0\bar R^{-4}\le\e_1$, we have
	\be{claimbound}
	\|z_\e(t)\|_{L^\infty(I)}\le C\bar R^{1/2},\quad -\infty< t<T^*_\e.
	\ee
	As a consequence, we have 
	$$T^*_\e=\infty,$$ 
	\be{claimboundH2}
	\|z_\e(t)\|_{H^2(I)}\le C\bar R^{1/2},\quad t\in\R
	\ee
	\be{globalboundz0stepGlobal}
	\int_0^\infty \bigl\{\|\partial_t z_\e(\tau)\|_{L^2(I)}^2 + \e \|\partial_t z_\e''(\tau)\|_{L^2(I)}^2\bigr\} d\tau \le
		C\bar R
	\ee
	and
	\be{globalboundz1stepGlobal}
	\int_0^\infty \big\{ \e^{1/4}\|\partial_t z_\e(\tau)\|_{L^\infty(I)}^2
		+\e^{3/4} \|\partial_t z_\e'(\tau)\|_{L^\infty(I)}^2\bigr\} d\tau \le C\bar R.
	\ee
	Furthermore, for each $\e\in(0,\e_0]$, we have
	\be{claimbound2}
	\sup_{t>0} \|z_\e(t)\|_{W^{4,\infty}(I)}<\infty. 
	\ee
\end{prop}

The key to the uniform estimate \eqref{claimbound} of $z_\e$ is the existence of a conserved quantity, given by the following past time average of $z$:
\be{Sndef}
\Theta_\e(t)=\int_{\tau=0}^\infty \int_Iz_\e(s,t-\e \tau)\rhobar(s,\tau)ds d\tau,\quad
\hbox{ where }\rhobar(s,\tau)=\int_\tau^\infty \rho(s,a)da,
\ee
for $t\in (0,T^*)$.
The function $\Theta_\e$ will be used in connection with the scalar product $(z(t),\rhostar)$,
which turns out to correspond with $\Theta_\e(t)$ {\it without time shift}.
The following lemma provides the crucial properties of $\Theta_\e$.
 We recall that the constants $\kappa_\e$ are defined in \eqref{defkappae}.

\begin{lem}  \label{lemunifbounds}
	Let the assumptions of Proposition~\ref{thmunifbounds} be in force.
	\smallskip
	
	(i) The function $\Theta_\e(t)$ is constant, namely
	\be{contphi0}
	\Theta_\e(t)=\kappa_\e,\quad t\in (0,T_\e^*),
	\ee
		
	(ii) We have
	\be{boundphipsi}
	\sup_{t\in(0,T_\e^*)} \bigl|(z_\e(t),\rhostar)-\kappa_\e\bigr| \le C(\bar R\e)^{1/2}. 
	\ee
\end{lem}

\begin{proof} (i) Integrating the equation 
	in space (i.e.,~taking $v=1$ as test-function in \eqref{mainweak}),  
	we obtain
	\be{ConservPpty}
	\int_I\int_0^\infty (z(s,t)-z(s,t-\e a))\rho(s,a)dads=0,\quad 0<t<T_\e^*.
	\ee
	Fix any $0<t<t'<T_\e^*$ and set $\ell=\e^{-1}(t'-t)$.  
	We denote $z(t)\cdot\rhobar(\tau)=\int_Iz(s,t)\rhobar(s,\tau)ds$ for conciseness.
	Putting $\tau=y+\ell$ in the definition of $\Theta(t')$, 
	 and denoting $y_+=\max(y,0)$, we obtain
	$$\begin{aligned}
		\Theta(t')-&\Theta(t)
		=\int_{-\ell}^\infty z(t-\e y)\cdot\rhobar(y+\ell)dy
		-\int_0^\infty z(t-\e \tau)\cdot\rhobar(\tau)d\tau\\
		&=\int_{-\ell}^\infty z(t-\e y)\cdot\bigl(\rhobar(y+\ell)-\rhobar(y_+)\bigr)dy
		+\int_{-\ell}^0 z(t-\e y)\cdot\rhobar(0)dy\\
		&=\int_{-\ell}^0\int_0^\infty z(t-\e y)\cdot\rho(a)dady-\int_{-\ell}^\infty\int_{y_+}^{y+\ell} z(t-\e y)\cdot\rho(a)dady
		\equiv I_1-I_2.
	\end{aligned}$$
	Noting that $I_2=\int_0^\infty\int_{a-\ell}^a z(t-\e y)\cdot\rho(a)dyda=\int_{-\ell}^0\int_0^\infty z(t-\e y-\e a)\cdot\rho(a)dady$
	by Fubini, and putting $\sigma=t-\e y$, it follows that
	$$\begin{aligned}
		\Theta(t')-\Theta(t)
		&= \int_{-\ell}^0\int_0^\infty \bigl(z(t-\e y)-z(t-\e y-\e a)\bigr)\cdot\rho(a)dady=\e^{-1}\int_t^{t'}
		\left(\int_0^\infty (z(\sigma)-z(\sigma-\e a))\cdot\rho(a)da\right)d\sigma,
	\end{aligned}$$
	hence $\Theta(t')=\Theta(t)$ 
	in view of \eqref{ConservPpty}.
	
	On the other hand, for each $s\in I$, $\tau>0$, we have $z(s,t'-\e \tau)\to z_p(s,-\e \tau)$ as $t'\to 0^+$,
	owing to  \eqref{hyppbm1locz0}.
	Moreover, for $t'\in(0, T_\e)$, we have
	$|z(s,t'-\e \tau)|\rho(s,a)\le\bigl(\sup_{\tau\in(-\infty, T_\e)}\|z(\tau)\|_\infty\bigr)\rho(s,a)\in L^1(I\times(0,\infty))$.
	Property \eqref{contphi0} follows by dominated convergence.
	\smallskip
	
	(ii) By \eqref{globalboundz0stepGlobal}
	we have
	\be{boundphipsi3}
	\begin{aligned}
		\bigl \|z(t)-z(t_0)\bigr\|_2
		&\le\int_{t_0}^t \|z_t(\sigma)\|_{2} d\sigma
		\le (t-t_0)^{1/2}\Bigl(\int_{t_0}^t \|z_t(\sigma)\|^2_2 d\sigma\Bigr)^{1/2}\le C(t-t_0)^{1/2}\bar R^{1/2},\quad 0<t_0<t<T_\e^*.
	\end{aligned}
	\ee
	Next, for $ -\infty<t_0<0<t<T_\e^*$, we may write
	$$\begin{aligned}
		\bigl \|z(t)-z(t_0)\bigr\|_2
		&\le \|z(t)-z(0^+)\|_2+\|z(0^+)-z_p(0)\|_2+\|z_p(0)-z_p(t_0)\|_2\le C\bar R^{1/2}t^{1/2}+C(R+\hat R)\e^{1/2}+C\hat R \,{\rm min}(1,|t_0|),
	\end{aligned}$$	
	owing to
	\eqref{boundphipsi3}, \eqref{zeps-discont} and  \eqref{hyppbm1locz0}.
	 Using ${\rm min}(1,|t_0|)\le |t_0|^{1/2}$, we obtain
	\be{boundphipsi2}
	\begin{aligned}
		\bigl \|z(t)-z(t_0)\bigr\|_2 
		&\le C\bar R^{1/2}((t-t_0)^{1/2}+\e^{1/2}),\quad -\infty<t_0<0<t<T_\e^*.
	\end{aligned}
	\ee
	Now, by the definition of $\rhobar, \rhostar$, the first part of assumption \eqref{hyp10}
	and  \eqref{contphi0}, we have 
		\be{boundphipsi00}
		\begin{aligned}
			\int_0^\infty\bigl(1+\tau^{1/2}\bigr)  \|\rhobar(\cdot, \tau) \|_2d\tau
			&\le\int_0^\infty\int_\tau^\infty \bigl(1+\tau^{1/2}\bigr)  \| \rho (\cdot,a) \|_2 dad\tau\le C \int_0^\infty \bigl(1+a^{3/2}\bigr)	 \| \rho (\cdot,a) \|_2 da<\infty
		\end{aligned}
		\ee
	
	and
	$$(z(t),\rhostar)-\kappa_\e=\int_0^\infty \int_I \bigl(z(s,t)-z(s,t-\e\tau)\bigr)\rhobar(s,\tau)dsd\tau,
	\quad  0<t<T_\e^*.$$ 
	Using 
	\eqref{boundphipsi3}-\eqref{boundphipsi00}, we obtain 
		\be{boundphipsi001}
		\begin{aligned}
			\bigl|(z(t),\rhostar)-\kappa_\e\bigr|& \leq \int_0^\infty \| z(t)-z(t-\e \tau) \|_2 \| {\varphi}(\cdot, \tau) \|_2 d\tau  \leq C(\bar R\e)^{1/2}\int_0^\infty \bigl(1+\tau^{1/2}\bigr) \|{\varphi}(\cdot, \tau) \|_2 d\tau,
			\quad   0<t<T_\e^*,
		\end{aligned}
		\ee
	hence \eqref{boundphipsi}.
\end{proof}

\begin{proof}[Proof of Proposition~\ref{thmunifbounds}]
	We claim that
	\be{zstn}
	\min_{s\in(0, L)}|z(s,t)|\le CR+C(\bar R\e)^{1/2}, \quad 0<t<T_\e^*.
	\ee
	If $z(s,t)=0$ for some $s\in (0, L)$, then there is nothing to prove.
	We may thus assume that $z(s,t)\ne 0$ for all $s\in(0, L)$.
	By continuity we may assume $z(\cdot,t)>0$ in $(0,L)$ (the case $z<0$ is similar).
	 Setting
		$$L_1:=\int_I\rhostar(s)ds=\int_I\int_{\tau=0}^\infty\rhobar(s,\tau) d\tau
		=\int_I\int_{a=0}^\infty a\rho(s,a) dads>0,$$
		and using \eqref{boundphipsi}, we obtain
		$$L_1\min_{s\in(0,L)}z(s,t)\le \int_Iz(s,t)\rhostar(s)ds\le |\kappa_\e|+C(\bar R\e)^{1/2},$$
		hence \eqref{zstn}.
		
		\smallskip
		
		It follows from \eqref{globalboundz1} and \eqref{zstn} that
		$$\|z(t)\|_\infty\le CR+C\bar R^{1/3}+C(\bar R\e)^{1/2}\le  C_3\bar R^{1/2}, \quad  0<t<T_\e^*,$$
		for some $C_3=C_3(\rho,L)>0$.
		This proves \eqref{claimbound}, hence \eqref{claimboundH2} 
		owing to estimates \eqref{globalboundz1} in Proposition~\ref{thm.dt.z.unif}.
	
	Now, we may choose $c_0>0$ depending only on $\rho, L$ such that
	$\e_0:=c_0\bar R^{-4}\le\e_1$ and $c_5\e_0^{-1/8}> C_3\bar R^{1/2}$. As a direct consequence of Theorem~\ref{thmloc}(iii),
	we deduce that $T_\e^*=\infty$.
	Properties \eqref{globalboundz0stepGlobal}, \eqref{globalboundz1stepGlobal} then follow from 
		\eqref{globalboundz0step}-\eqref{Energy-init2}.
	
	\smallskip
	
	Finally, using \eqref{globalboundz1}, \eqref{claimbound} and the differential equation \eqref{mainstrong1},
	we get 
	$$
	\sup_{t>0} \|z_\e(t)\|_{H^4(I)}<\infty,
	$$
	hence in particular 
	$\sup_{t>0} \|z_\e(t)\|_{W^{2,\infty}(I)}<\infty$. 
	Going back to \eqref{mainstrong1} we obtain \eqref{claimbound2}.
\end{proof}

\section{H\"older continuity with respect to $\e$ (proof of Theorem~\ref{thmglob}{\rm (ii)})} \label{SecCont}

It is a direct consequence of:

\begin{prop} \label{prop-conteps}
	Assume \eqref{hyppbm1loc}--\eqref{hyp10}, \eqref{hyppbm1locz0},
	and let $\e_0$ be given by Proposition~\ref{thmunifbounds}.
	We have
	\be{conteps0}
	\|(z_\e-z_{\bar\e})(t)\|_{H^2} \le C\bar\e^{-1/4} \bar R({\bar\e}-\e)^{1/4}\exp\bigl\{\bar\e^{-1}\|\rho\|_\infty t\bigr\}. 
	\quad 0<\e<\bar\e<\e_0,
	\ee
\end{prop}

\begin{proof}[Proof of Proposition~\ref{prop-conteps}]
		{\bf Step 1.} {\it First estimate of the difference.}
		Let $0<\e<\bar\e<\e_0$ and $t>0$. We claim that 
		\be{conteps1}
		\begin{aligned}
			{\bar\e} \int_I &(z''_\e-z''_{\bar\e})^2(t) ds+\int_I (z_\e-z_{\bar\e})^2(t) ds \le C\bar R^2({\bar\e}-\e)+C\int_I \int_0^\infty (z_\e-z_{\bar\e})^2(t-\bar\e a)\rho(s,a)dads \\
			&+C\int_I \int_0^\infty (z_\e(t-\e a)- z_\e(t-{\bar\e} a))^2\rho(s,a) dads.
		\end{aligned}
		\ee
		
		To this end, subtracting the equation \eqref{mainweak} for $z_\e$ and for $z_{\bar\e}$
		and choosing $v=(z_\e-z_{\bar\e})(t)$, we obtain
		(omitting the variable $s$ without risk of confusion)
		\be{conteps1a}
		\begin{aligned}
			&{\bar\e} \int_I (z''_\e-z''_{\bar\e})^2(t) ds+ \int_I {\mu} (z_\e-z_{\bar\e})^2(t) ds 
			={\bar\e}
			\int_I (F'(z'_\e)-F'(z'_{\bar\e}))(t) (z_\e-z_{\bar\e})'(t) ds \\
			&\quad+({\bar\e}-\e) \left\{\int_I z''_\e(t) (z_\e-z_{\bar\e})''(t)ds-\int_I (F'(z'_{\e}(t)))(z_\e-z_{\bar\e})'(t)\right\}ds\\
			&\quad+\int_I \int_0^\infty (z_\e-z_{\bar\e})(t-\bar\e a)\rho(s,a) (z_\e-z_{\bar\e})(t) dads +\int_I \int_0^\infty (z_\e(t-\e a)- z_\e(t-{\bar\e} a))\rho(s,a) (z_\e-z_{\bar\e})(t) dads.
		\end{aligned}
		\ee
		To estimate the first term on the right hand side of \eqref{conteps1a}, we use \eqref{globalboundz1},
		\eqref{claimboundH2} and the Sobolev embedding to write
		$$
		\left| F'(z'_\e)-F'(z'_{\bar\e})\right| \leq 4\left\{ \left| z'_\e - z_{\bar\e}' \right| + 
		\left| (z'_\e)^3 - (z_{\bar\e}')^3\right|\right\} \leq C\bar R\left| z'_\e - z_{\bar\e}'\right| .
		$$
		Next,  using \eqref{interpol-eta}, one writes as above:
		$$
			 \left| \int_I (F'(z'_\e)-F'(z'_{\bar\e})) (z'_\e-z'_{\bar\e}) ds\right| 
			  \leq 
			C\bar R
			\int_I \left| z'_\e(s,t) - z_{\bar\e}'(s,t)\right|^2 ds 
			\leq  C_4\bar R
			\left\{\delta \| z_\e''(t) - z_{\bar\e}''(t)\|_2^2+ (1+\delta^{-1})  \| z_\e(t) - z_{\bar\e}(t)\|_2^2 \right\}
		$$
		with $C_4=C_4(\rho,L)>0$.
		Choosing $\delta=(2C_4\bar R)^{-1}$, using $\bar R^2\bar\e\le c_0\bar R^{-2}\le c_0$
			and taking $c_0=c_0(\rho,L)>0$ smaller in Proposition~\ref{thmunifbounds} if necessary, we obtain
		$$\begin{aligned}
			{\bar\e}& \left| \int_I (F'(z'_\e(t))-F'(z'_{\bar\e}(t)))(z_\e-z_{\bar\e})'(t)ds\right|\le C\bar R^2\bar\e \int_I (z_\e-z_{\bar\e})^2(t) ds+\frac{{\bar\e}}{2}\int_I (z''_\e-z''_{\bar\e})^2(t) ds\\
			&\qquad\le \frac{\mu_{min}}{2} \int_I (z_\e-z_{\bar\e})^2(t) ds+\frac{{\bar\e}}{2}\int_I (z''_\e-z''_{\bar\e})^2(t) ds.
		\end{aligned}$$
		Next using 
		\eqref{claimboundH2}, we can estimate the third term on the right hand side of \eqref{conteps1a} by
		$$
		({\bar\e}-\e) \int_I \bigl|z''_\e(t) (z_\e-z_{\bar\e})''(t)\bigr|ds+\int_I \bigl|F'(z'_{\e}(t))(z_\e'-z_{\bar\e}')(t)\bigr|ds 
		\le C\bar R^2({\bar\e}-\e).$$
		Estimating the remaining two terms by Young's inequality, we get \eqref{conteps1}.
		
		\smallskip
		{\bf Step 2.} {\it Control of the time shift term in the right hand side of \eqref{conteps1}.}
		We claim that
		\be{conteps2}
		\int_I \int_0^\infty \bigl|z_\e(t-\e a)- z_\e(t-{\bar\e} a)\bigr|^2\rho(s,a) dads
		\le C\bar R\bigl(\bar\e({\bar\e}-\e)\bigr)^{1/2}.
		\ee
		To this end, recalling that $z_\e(\tau)=z_{\bar\e}(\tau)=z_p$ for $\tau<0$, we split the left hand side of \eqref{conteps2} as
		$$\begin{aligned}
			\int_I &\int_0^\infty  \bigl|z_\e(t-\e a)- z_\e(t-{\bar\e} a)\bigr|^2\rho(s,a) dads 	=\int_I \int_0^{t/{\bar\e}}  \bigl|z_\e(t-\e a)- z_\e(t-{\bar\e} a)\bigr|^2\rho(s,a) dads \\
			& \qquad+\int_I \int_{t/{\bar\e}}^{t/\e}  \bigl|z_\e(t-\e a)- z_\e(t-{\bar\e} a)\bigr|^2\rho(s,a) dads.
		\end{aligned}$$
		By \eqref{boundphipsi3}, we have
		\be{conteps2a}
		\int_I \int_0^{t/{\bar\e}}  \bigl|z_\e(t-\e a)- z_\e(t-{\bar\e} a)\bigr|^2\rho(s,a) dads
		\le C({\bar\e}-\e)\bar R\int_I \int_0^\infty a \rho(s,a) dads.
		\ee
		On the other hand, by \eqref{boundphipsi2}, we have
		$$
		\int_I \bigl|z_\e(t-\e a)- z_\e(t-{\bar\e} a)\bigr|^2 ds\le C\bar R\bigl((\bar\e-\e)a+\e\bigr),\quad t/\bar\e<a<t/\e.$$
		Since, using the second part of assumption \eqref{hyp10},
		$$\begin{aligned}
			\int_{t/{\bar\e}}^{t/\e}\|\rho(\cdot,a)\|_\infty da
			&\le C\min\left\{\frac{\bar\e}{t} \int_0^\infty a \|\rho(\cdot,a)\|_\infty da,
			\frac{t({\bar\e}-\e)}{\e{\bar\e}}\|\rho\|_\infty \right\} \le C\min\left\{\frac{\bar\e}{t},\frac{t({\bar\e}-\e)}{\e{\bar\e}}\right\}
			\le C\Bigl(\frac{{\bar\e}-\e}{\e}\Bigr)^{1/2},
		\end{aligned}$$
		it follows that
		$$\begin{aligned}
			\int_I &\int_{t/{\bar\e}}^{t/\e}  \bigl|z_\e(t-\e a)- z_\e(t-{\bar\e} a)\bigr|^2\rho(s,a) dads\le  \int_{t/{\bar\e}}^{t/\e} C\bar R\bigl((\bar\e-\e)a+\e\bigr) \|\rho(\cdot,a)\|_\infty da\\
			&\le  C\bar R(\bar\e-\e)\int_0^\infty a \|\rho(\cdot,a)\|_\infty da
			+ C\bar R\e\int_{t/{\bar\e}}^{t/\e}  \|\rho(\cdot,a)\|_\infty da\le  C\bar R(\bar\e-\e)
			+ C\bar R\e\Bigl(\frac{{\bar\e}-\e}{\e}\Bigr)^{1/2}\le C\bar R\bigl(\bar\e({\bar\e}-\e)\bigr)^{1/2}.
		\end{aligned}$$
		Claim \eqref{conteps2} follows by adding this with \eqref{conteps2a}.

	\smallskip
	{\bf Step 3.} {\it Conclusion.}
	Letting $\delta=C\bar R^{2}\bigl(\bar\e({\bar\e}-\e)\bigr)^{1/2}$,
	$$\Lambda_{\bar\e,\e}(t)={\bar\e} \int_I (z''_\e-z''_{\bar\e})^2(t) ds+\int_I (z_\e-z_{\bar\e})^2(t) ds,$$
	combining \eqref{conteps1}, \eqref{conteps2}, 
	and using that $z_\e(\tau)=z_{\bar\e}(\tau)=z_p$ for $\tau<0$, 
	we obtain
	\be{conteps3}
	\begin{aligned}
		\Lambda_{\bar\e,\e}(t)
		&\le \delta+C\int_0^{t/\bar\e} \Lambda_{\bar\e,\e}(t-\bar\e a)\|\rho\|_\infty da
		\le \delta+C\bar\e^{-1}\int_0^t \Lambda_{\bar\e,\e}(\tau)\|\rho\|_\infty d\tau.
	\end{aligned}
	\ee
	It follows from Gronwall's lemma that
	$\Lambda_{\bar\e,\e}(t)\le \delta\exp\bigl\{\bar\e^{-1}\|\rho\|_\infty t\bigr\}$,
	hence \eqref{conteps0}.
\end{proof}

\section{Proof of convergence as $t\to\infty$ (Theorem~\ref{thmglob}{\rm (iii)})}\label{SecConv}

We denote by $\mathcal{S}$  the set of steady states, i.e.~solutions of \eqref{eqstatz}.
The $\omega$-limit set of $z_\e$ is defined by
\be{defomegaze}
\omega(z_\e):=\left\{Z\in H^2(I);\, \exists t_n\to\infty,\ 
\lim_n \|z_\e(t_n)-Z\|_{H^2(I)}=0\right\}.
\ee
Also, recalling \eqref{Sndef2}, 
we define
$$\mathcal{S}_K:=\bigl\{Z\in \mathcal{S};\ (Z,\rhostar)=K\bigr\},\quad K\in\R.$$
For fixed $\e\in(0,\e_0]$, the convergence as $t\to\infty$ will be a direct consequence
of the following quasiconvergence property, along with the structure of steady states.

\begin{lem} \label{prop-quasistat}
	Let the assumptions of Theorem~\ref{thmglob} be in force and let $\e\in(0,\e_0]$.
	\smallskip
	
	(i) The set $\omega(z_\e)$ is a nonempty compact connected subset of $H^2(I)$.
	
	\smallskip
	
	(ii) We have $\omega(z_\e)\subset \mathcal{S}_{\kappa_\e}$.
\end{lem}

Whereas the proof of Lemma~\ref{prop-quasistat}(i) is standard, the proof of assertion (ii) relies on two ingredients: the decay of the delay term $\cL_\e [z_\e]$ as $t\to\infty$
and the good properties of the inner product of $z_\e(t)$ with $\rhostar$.

\begin{lem} \label{lem-quasistat}
	Let the assumptions of Theorem~\ref{thmglob} be in force.
	\begin{itemize}
		\item[(i)]Let $p\in[1,\infty)$. The function $\cL_\e [z_\e]$ satisfies the decay property
		\be{quasistat1}
		\lim_{t\to\infty}\|\cL_\e [z_\e](\cdot,t)\|_{L^p(I)}=0,\quad\hbox{for each $\e\in(0,\e_0]$.}
		\ee
		
		\item[(ii)]We have the convergence
		\be{quasistat1b}
		\lim_{t\to\infty}(z_\e(t), \rhostar)=\kappa_\e,\quad\hbox{for each $\e\in(0,\e_0]$.}
		\ee
	\end{itemize}
\end{lem}

\begin{proof}
	(i) We have
	$$\int_I |\cL_\e z(s,t)|ds\le \int_{a=0}^\infty \int_I H_\e(t,s,a) ds da,\quad\hbox{ where } H_\e(t,s,a):=\e^{-1}|z(s,t)-z(s,t-\e a)|\rho(s,a).$$
	For each fixed $(s,a)\in I\times(0,\infty)$, 
	inequality \eqref{globalboundz1stepGlobal} guarantees
	that, for all $\e\in(0,\e_0]$ and $t>\e a$, 
	\be{estimHeps}
	H_\e(t,s,a)\le \sqrt{\frac{a}{\e}} \,\rho(s,a)\Bigl(\int_{t-\e a}^t \|z_t(\tau)\|^2_\infty d\tau\Bigr)^{1/2}
	\le C\bar R^{1/2}\e^{-5/8}\sqrt{a}\rho(s,a).
	\ee
	Moreover, the first inequality in \eqref{estimHeps} implies $\lim_{t\to\infty}H_\e(t,s,a)=0$ for each fixed $\e,s,a$.
	Since $\sqrt{a}\rho(s,a)\in L^1(I\times(0,\infty))$, 
	by dominated convergence, we obtain \eqref{quasistat1} for $p=1$. 
	Since, on the other hand,
	$\|\cL_\e z(\cdot,t)\|_\infty\le C\bar R^{1/2}\e^{-1}\|\rho\|_{L^1(I\times(0,\infty))}$ by \eqref{claimbound},
	we deduce \eqref{quasistat1} for all finite $p$ by interpolation.
	
	\smallskip

	(ii) Fix $\e\in(0,\e_0]$.
	 Recalling \eqref{boundphipsi001} we have
		$$\bigl|(z(t),\rhostar)-\kappa_\e\bigr| \leq \int_0^\infty \|z(t)-z(t-\e \tau) \|_2 \| {\varphi}(\cdot, \tau) \|_2 d\tau$$
		and, for each fixed $\tau>0$,  \eqref{globalboundz1stepGlobal} guarantees that
		$$\|z(t)-z(t-\e \tau) \|_2\le \sqrt{\e\tau} \Bigl(\int_{t-\e\tau}^t \|z_t(\sigma)\|^2_2d\sigma\Bigr)^{1/2}\to 0,\, \text{as}\,\to\infty.$$
		On the other hand, we have  $\|z(t)-z(t-\e \tau) \|_2 \| \varphi(\cdot, \tau) \|_2 \le 
		C\bar R \|\varphi(\cdot, \tau) \|_2\in L^1_\tau(0,\infty)$, owing to \eqref{claimbound} and \eqref{boundphipsi00}.
	Property \eqref{quasistat1b} then follows by dominated convergence.
\end{proof}

\begin{proof}[Proof of Lemma~\ref{prop-quasistat}]
	(i) It is easy to verify that
	$$\omega(z_\e)= {\mathop{\capsize{\cap}}_{n\in\N^*}}
	\overline{K_n},\quad\hbox{where } K_n=\{z_\e(t);\ t\ge n\}.$$
	On the other hand, it follows from \eqref{claimbound2} that each $\overline{K_n}$ is compact
	(for the $H^2(I)$ topology). Moreover, since $z\in C((0,\infty);H^2(I))$ as a consequence of \eqref{zLip0}, 
	$K_n$ hence $\overline{K_n}$ is connected.
	We conclude that $\omega(z_\e)$ is a nonempty compact connected set, 
	as the intersection of a nonincreasing sequence of such sets.

	\smallskip
	
	(ii) Let $\e\in(0,\e_0]$. For any $v\in H^2(I)$ and $t>0$, we have
	\be{weaksolL2}
	\begin{aligned}
		\Bigl|\int_I \bigl(\zeps'' (s,t) v'' (s) + F'(\zeps'(s,t)) v'(s)\bigr)ds\Bigr| 
		&= \Bigl|\int_I\cL_\e z_\e (s,t)   v(s) ds\Bigr|\le \|\cL_\e  z_\e (\cdot,t)\|_{L^2(I)}\|v\|_{L^2(I)}.
	\end{aligned}
	\ee
	Let $Z\in\omega(z_\e)$.
	There exists a sequence $t_n\to\infty$ such that
	$\lim_n \|z_\e(t_n)-Z\|_{H^2(I)}=0$. 
	Passing to the limit in \eqref{weaksolL2} for $t=t_n$
	with help of \eqref{quasistat1} with $p=2$, it follows that
	$$\int_I \bigl(Z'' (s) v'' (s)  + F'(Z'(s)) v'(s)\bigr)ds=0.$$
	Consequently, $Z\in\mathcal{S}$.
	Taking $\kappa_\e$ given by Lemma~\ref{lem-quasistat}, property \eqref{quasistat1b} then guarantees that 
	 $(Z,\rhostar)=\kappa_\e$, hence $Z\in\mathcal{S}_{\kappa_\e}$.
\end{proof}

\begin{proof}[Proof of Theorem~\ref{thmglob}(iii)]
	The set $\omega(z_\e)$ is nonempty and connected by Lemma~\ref{prop-quasistat} (i).
	On the other hand, it is finite since $\omega(z_\e)\subset\mathcal{S}_{\kappa_\e}$ by 
	Lemma~\ref{prop-quasistat}(ii) and 
	$\mathcal{S}_{\kappa_\e}$ is finite by Proposition~\ref{propeqstatw2}.
	The set $\omega(z_\e)$ is thus a singleton, which proves the theorem.
\end{proof}

\section{Convergence when $\e$ goes to $0$ (proof of first part of Theorem~\ref{thm-cvE}{\rm (i)})}
\label{SecConvEps}

 The result will be a consequence of the following two lemmas. 

\begin{lem}\label{lem.eps.zero} 
	Assume \eqref{hyppbm1loc}--\eqref{hyp10}, \eqref{hyppbm1locz0} 
	and pick any sequence $\e_j\to 0^+$.
	\smallskip
	
	(i) There exists a subsequence $\e'_j$ and
	\be{class.zz.weak}
	\tilde z_0\in L^\infty([0,\infty); H^2(I))
	\cap\, C^\nu_{loc}([0,\infty); C^1(\overline I))
	\cap\, H^1_{loc}([0,\infty); L^2(I))
	\ee
	for all $\nu\in(0,1/8)$, such that, $\lim_j z_{\e'_j} = \tilde z_0$
	where, for each $T>0$, the convergence is strong in $C([0,T];$ $C^1(\overline I)$, weak in $H^1((0,T);L^2(I))$, 
	and weak-* in $L^\infty(0,T;H^2(I))$.
	
	\smallskip
	
	(ii) The function $\tilde z_0$ is a global solution of \eqref{mainstrong1-0}
	(cf.~Definition~\ref{def-auxilf}) with $b=\rhostar$ and $\tilde z(0) = \zp(0)$.
\end{lem}

\begin{proof}
	(i)  As a consequence of estimates \eqref{claimboundH2}, \eqref{globalboundz0stepGlobal},
		similar to \cite[Lemma 5, p.976]{OeSch}, we see that $\{z_\e\}$ is bounded in $L^\infty([0,\infty); H^2(I))\cap H^1_{loc}([0,\infty); L^2(I))$
		hence, by interpolation, it is bounded in
		$W^{\theta,2/\theta}_{loc}([0,\infty);H^{2(1-\theta)}(I))\hookrightarrow $ $ C^{(\theta-\eta)/2}_{loc} $ $([0,\infty);$ $ C^1(\overline I))$
		for all $\theta\in(0,1/4)$ and $\eta>0$.
		The conclusion then follows from standard compactness properties.
	
	\smallskip
	
	(ii) 
	Let $\e:=\e'_j$ be as in assertion (i) and fix $\psi \in C^1([0,T];H^2(I))$. 
		Taking $v=\psi(\cdot,t)$ in \eqref{mainweak} and integrating in time, we obtain
		\be{testzepsi}
		{\mathcal{D}}_{\varepsilon}(\rho,{z}_{\varepsilon},\psi)+\int_0^T\int_I \bigl(\zeps''\psi''+ F'(\zeps') \psi' \bigr)ds dt =0,
		\ee
		where
	$$
	{\mathcal{D}}_{\varepsilon}(\rho,{z}_{\varepsilon},\psi):={\frac{1}{\varepsilon}}\int_{I} 
	\int_{0}^{T}
	 \int_0^\infty
	\rho(s,a)\left({z}_{\varepsilon}(s,t)-{z}_{\varepsilon}(s,t-\varepsilon a)\right)d a\psi(s,t)d t d s.
	$$
	With the decomposition  
	$$
	\begin{aligned}
		{\mathcal{D}}_{\varepsilon}& (\rho,{z}_{\varepsilon},\psi) ={\frac{1}{\varepsilon}}\int_{I}\int_{0}^{T}\int_{\frac{T-t}{\varepsilon}}^{\infty}\rho(s,a){z}_{\varepsilon}(s,t)\psi(s,t)d a d t d s 
	 -\;\frac{1}{\varepsilon}\int_{I}\int_{0}^{T}\int_{0}^{\frac{T-t}{\varepsilon}}{z}_{\varepsilon}(s,t)(\psi(s,t+\varepsilon a)-\psi(s,t))\,\rho(s,a)d a d t d s \\
		& -\;\frac{1}{\varepsilon}\int_{I}\int_{0}^{T}\int_{\frac{t}{\varepsilon}}^{\infty}\rho(s,a)z_{p}(s,t-\varepsilon a)\psi(s,t)d a d t d s 
		=:  {\mathcal{D}}^1_\e- {\mathcal{D}}^2_\e-{\mathcal{D}}^3_\e,
	\end{aligned}
	$$
	we use an argument similar to \cite[Propositions~A.1 and A.2.~pp.41-42]{Mi.5}. 
	 Namely, using the change of variables $(t,a)=(T-\e h,a)$ (resp.,  $(t,a)=(\e(a-\tau),a)$), 
		the convergence $z_\e\to \tilde z_0$ in $L^\infty(I\times(0,T))$, the regularity of $\psi$
		and assumptions \eqref{hyp10}, \eqref{hyppbm1locz0}, it follows by dominated convergence that
		$$	\begin{aligned}
			{\mathcal{D}}^1_\e&=\int_{I}\int_{0}^{T/\e}\int_h^{\infty}\rho(s,a)(z_\e \psi)(s,T-\e h)dadhds \\
			& \xrightarrow[j\to\infty]{} \int_{I}\int\int_{\{0<h<a<\infty\}}\rho(s,a)(\tilde\zz \psi)(s,T)dadhds=\int_I (\tilde \zz\psi)(s,T)\rhostar(s) ds,\\
			{\mathcal{D}}^3_\e&=\int_{I}\int_{0}^\infty \int_0^a \rho(s,a)z_p(s,-\e\tau) \psi(s,\e(a-\tau))d\tau dads \\
			&\xrightarrow[j\to\infty]{}\int_{I}\int\int_{\{0<\tau<a<\infty\}}\rho(s,a)(z_p \psi)(s,0)d\tau dads=\int_I (z_p \psi)(s,0)\rhostar(s) ds,\\
			{\mathcal{D}}^2_\e& \xrightarrow[j\to\infty]{} \int_I\int_0^T \int_0^\infty z_0(s,t)  \dt  \psi(s,t)a \rho(s,a) da dt ds
			=\int_0^T \int_I  \tilde z_0 \dt \psi \rhostar  ds dt.
		\end{aligned}$$
		Therefore,
	$$
	\begin{aligned}
		{\mathcal{D}}_{\varepsilon}(\rho,{z}_{\varepsilon},\psi) \xrightarrow[j\to\infty]{} &\int_I\tilde \zz(T)\psi(T)\rhostar ds
		- \int_0^T \int_I   \tilde z_0 \dt \psi \rhostar  ds dt - \int_I z_p(0) \psi(0) \rhostar ds \\
		& = \int_0^T \int_I \dt \tilde{z}_0 \psi \rhostar ds dt
		+ \int_I (\tilde{z}_0-z_p)(0) \psi(0) \rhostar ds ,
	\end{aligned}
	$$
	where  the	 integration by parts in time is allowed since
	$\tilde{z}_0 \in H^1(0,T;L^2(I))$. 
	On the other hand, thanks to the weak and strong convergence  properties in assertion (i),
		we may pass to the limit in the integral term in \eqref{testzepsi}, and we get 
		$$ \int_0^T \int_I \dt \tilde{z}_0 \psi \rhostar ds dt
		+ \int_I (\tilde{z}_0-z_p)(0) \psi(0) \rhostar ds+\int_0^T\int_I \bigl(\tilde \zz''\psi'' + F'(\tilde \zz') \psi' \bigr)ds dt=0.$$
		By density, this remains true for all $\psi \in L^2(0,T);H^2(I))$.
		Now, for any $\theta\in H^2(I)$, by choosing $\psi_n(s,t)=(1-nt)_+\theta(s)$ and letting $n\to\infty$, it follows that $\int_I (\tilde{z}_0-z_p)(0) \theta \rhostar ds=0$,
		hence $\tilde{z}_0(0)=z_p(0)$ and  $\tilde z_0$ is a global solution of \eqref{mainstrong1-0}.
\end{proof}

\begin{lem}\label{lem.unique-weak}
	 Under the assumptions of Proposition~\ref{prop-pbm0},
	problem \eqref{mainstrong1-0} has at most one solution in
	the class $X:=L^\infty_{loc}([0,\infty);H^2(I))\cap H^1_{loc}([0,\infty);L^2(I))$.
\end{lem}

\begin{proof}
	Let $\tilde z_0,  \bar z_0\in X$ be two solutions (cf.~Definition~\ref{def-auxilf}). Then, for each $T>0$, subtracting the equations,
	using $|F'(\tilde z_0')-F'(\bar z_0')|\le C(T)|w'|$ and taking $\psi=w:=\tilde z_0-\bar z_0$, we see that $w$ satisfies
	$$		\int_0^t \int_I  b(s) w\dt w ds d\tau 
	\le C(T)\int_0^t \int_I (w')^2ds d\tau-\int_0^t \int_I (w'')^2ds d\tau,\quad 0<t<T.$$
	Using the interpolation inequality \eqref{interpol-eta} and $w(0)=0$,
	we obtain,
	for all $t\in(0,T)$,
	$$	
	\begin{aligned}
			 \int_I  b(s) & w^2(t) ds = 2\int_0^t \int_I  b(s) w\dt w ds d\tau
		\le C_1(T)\int_0^t \int_I w^2 ds d\tau\le C_2(T)\int_0^t \int_I  b(s)w^2 ds d\tau.
	\end{aligned}
	$$
	By Gronwall's Lemma, we deduce that $w\equiv 0$.
\end{proof}

\begin{proof}[Proof of first part of Theorem~\ref{thm-cvE}(i)]
	Let $\tilde z_0$, in the class \eqref{class.zz.weak}, be any cluster point of the family $\{z_\e,\ \e\in(0,\e_0]\}$ 
	for the notion of convergence in Lemma \ref{lem.eps.zero}.
	By that lemma, there exists at least one, and $\tilde z_0$ is a 
	global solution of \eqref{mainstrong1-0} with $b=\rhostar$.
	
	On the other hand, by \eqref{hyppbmloc10}, setting $a_0=(2\|\rho\|_\infty)^{-1}\mu_{\min}$, we get
		$$ \rhostar(s)=\int_0^\infty a\rho(s,a)da\ge a_0\int_{a_0}^\infty \rho(s,a)da
		\ge a_0\bigl(\mu_{\min}-a_0\|\rho\|_\infty\bigr)\ge \frac{\mu_{\min}^2}{4\|\rho\|_\infty}>0,
		\ \ s\in I.$$
		It follows from Lemma~\ref{lem.unique-weak} that   \eqref{mainstrong1-0} 
		with $b=\rhostar$ has at most one global weak solution,
		hence the cluster point is unique, which implies the desired convergence as $\e\to 0$.
\end{proof}

\section{Proof of Theorem~\ref{thmglob}{\rm (iv)} and end of proof of Theorem~\ref{thm-cvE}: 
	strong $H^2$ convergence and stablity of affine steady states with respect to $\e$}

\label{SecStab}

The energy associated with the solution $z_0$ of \eqref{mainstrong1-0} is given by
$$
E_0(t) :=\frac12\int_{I} (z_0''(t))^2 ds+ \int_{I} F(z_0'(t)) ds.
$$
We shall use the following properties (see~\ \ref{Secz0} for the proof).

	\begin{prop}\label{prop-pbm0bEn}
		Let the assumptions of Proposition~\ref{prop-pbm0} be in force.
		
		\vskip 1mm
		
		(i) For all $t_2>t_1\ge 0$, we have 
		$$
		E_0(t_2)-E_0(t_1)=-\int_{t_1}^{t_2} \int_I b(\partial_tz_0)^2dsdt \le 0.
		$$
		In particular, $E_0\in W^{1,p}_{loc}([0,\infty))$ for all $p\in(1,\infty)$ and 
		$$E'_0(t)=-\int_I b(\partial_tz_0)^2ds,\quad a.e.~t>0.$$
		
		(ii) For all $t>0$, we have $(z_0(t),b)=K_0:=(\phi,b)$.
	\end{prop}

Our next result, 
which requires the additional coercivity assumption \eqref{hyprhoa} on the kernel,
yields the strong $L^2_t(H^2_s)$ convergence statement in Theorem~\ref{thm-cvE}(i).
It also provides the energy convergence property which will be the key to the proof of Theorem~\ref{thm-cvE}(ii).
We recall that $E_\e$ and $u_\e$ are defined in \eqref{defusta}-\eqref{defEeps}.

\newcommand{\cM}{{\mathcal M}}

\begin{prop}\label{prop-cvE}
	 Assume \eqref{hyppbm1loc}--\eqref{hyppbm1locz0}.
	Then, for all $T>0$, we have
	\be{concl-prop-cvE1}
	\lim_{\e\to 0} z_\e=z_0\ \hbox{ strongly in $L^2(0,T;H^2(I))$}
	\ee
	and
	\be{concl-prop-cvE2}
	\lim_{\e\to 0} E_\e(t)=E_0(t),\quad\hbox{uniformly for $t>0$ bounded.}
	\ee
	Moreover, 
	one has
	\be{concl-prop-cvE3}
	\lim_{\e\to 0} E'_\e(t)\equiv\lim_{\e\to 0} \frac{1}{2\e^2} \int_I \intrp u^2_\e(s,a,t) \da \kernel (s,a) da ds=
	-\int_I  \mu_1(s) \left| \dt \zz (s,t)\right|^2ds \equiv E'_0(t),
	\ee
the convergence taking place weakly in $\cM(0,T)$, {\em i.e.} in the space of finite Radon measures on $(0,T)$.
\end{prop}

\begin{rem}\label{rmk.10.1}
	(i) 
	The equality in \eqref{concl-prop-cvE3} can be explained heuristically as follows:
	when $\e \to 0$, one has formally that
	$$
	\frac{1}{2\e^2} \intrp \da \rho(s,a) \veps^2(s,a,t) da \to \frac{1}{2} \intrp \da \rho(s,a) u_0^2(s,a,t) da,
	\quad\hbox{where $u_0 (s,a,t):= a \dt \zz (s,t)$.}
	$$
	 Then the identity 
	$\partial_a  (a^2\rho) = 2a \rho +  a^2\da \rho$,
	because the integral of the left hand side vanishes, implies that $\mu_1(s)\equiv \intrp a \rho(s,a) da = - \frac12\intrp \da \rho (s,a) a^2 da$, 
	giving  the desired limit. 
	
	\smallskip
	
	(ii) With some additional work, it can be shown that the same result holds in the finite dimensional case \cite{MiOel.2}, as well as for the linear 
	delayed heat equation \cite{MiOel.4} and also in the case of delayed harmonic maps \cite{Mi.5}.
\end{rem}

The proof is based on the following two lemmas.
The first one is a higher order a priori estimate,
uniform for $\e>0$:

\begin{lem}\label{lem.compact0}
	Assume \eqref{hyppbm1loc}--\eqref{hyppbm1locz0}.
	Then, for each $T>0$, we have
	\be{concl-compact00}
	\sup_{\e\in(0,\e_0]} \|z_\e\|_{L^2(0,T;H^3(I))}<\infty
	\ee
	and
	\be{concl-compact0}
	\sup_{\e\in(0,\e_0]} \|z_\e\|_{H^{1/6}(0,T;H^{5/2}(I))}<\infty.
	\ee
	In particular, the family $\{z_\e,\ \e\in(0,\e_0]\}$ is precompact in $L^2(0,T;H^2(I))$.
\end{lem}

Our second lemma shows that the first term in the energy $E_\e$ vanishes as $\e\to 0$.

\begin{lem}\label{lem.compact1}
	Assume \eqref{hyppbm1loc}--\eqref{hyp10}, \eqref{hyppbm1locz0} and let
	$$G_\e(t):=\e^{-1}\int_I\int_0^\infty u_\e^2(t,s,a)\rho(s,a) dads.$$
	Then we have
	\be{res.compact1}
	\int_T^{T+1}G_\e(t) dt\le C\bar R\bigl(\eta(\e)+\sqrt{\e}\bigl(1+T^{-1/2}\bigr)\bigr),\quad\hbox{for all } T>0,
	\ee
	where $\eta(\e):=\int_{1/\sqrt{\e}}^\infty a\|\rho(\cdot,a)\|_\infty da\to 0$, as $\e\to 0$.
\end{lem}

\begin{proof}[Proof of Lemma~\ref{lem.compact0}]
	 Recalling \eqref{regulhigher}, multiplying  \eqref{mainstrong1} with $-z''$, integrating by parts, using \eqref{globalboundz1} and the Sobolev inequality $\|z''\|_\infty\le C(\|z''\|_2+\|z'''\|_2)$, we get, for all $t>0$,
	$$\begin{aligned}
		\|z'''\|_2^2
		&=-\int_I (F'(z'))'z''+\int_I (\mathcal{L}_\e z)z''\le C\int_I(1+(z')^2)(z'')^2 + \|\mathcal{L}_\e z\|_1\|z''\|_\infty\\
		&\le C\bar R^{5/3} + C\|\mathcal{L}_\e z\|_1(\|z'''\|_2+\bar R^{1/2})
		\le C\bar R^{5/3} + \frac12\|z'''\|_2^2+C\bar R+C\|\mathcal{L}_\e z\|_1^2,
	\end{aligned}$$
	hence
	\be{res.compact3}
	\|z'''\|_2^2\le C\bar R^{5/3} +C\|\mathcal{L}_\e z\|_1^2.
	\ee
	On the other hand, by Cauchy-Schwarz, we have
	$$\begin{aligned}
		\|\mathcal{L}_\e z\|_1^2
		&\le \e^{-2}\left(\int_I\int_0^\infty u(s,t,a)\rho(s,a) dads\right)^2= \e^{-2}\left(\int_I\int_0^\infty \bigl(u(s,t,a)|\rho_a(s,a)|^{1/2})\rho(s,a)|\rho_a(s,a)|^{-1/2}dads\right)^2\\
		&\le \e^{-2}\int_I\int_0^\infty u^2(t,s,a)|\rho_a(s,a)| dads \int_I\int_0^\infty \rho^2(s,a)|\rho_a(s,a)|^{-1} dads
	\end{aligned}$$
	(where we defined $\rho(s,a)|\rho_a(s,a)|^{-1/2}:=0$ if $\rho(s,a)=\rho_a(s,a)=0$).
	Consequently, by~\eqref{hyprhoa}, \eqref{eq.dissipation} and \eqref{Energy-init2},
	$$\int_0^T\|\mathcal{L}_\e z\|_1^2dt\le \e^{-2}\int_0^T \int_I\int_0^\infty u^2(t,s,a)|\rho_a(s,a)| dads\le C\bar R.$$
	From this, \eqref{globalboundz1} and \eqref{res.compact3}, we deduce \eqref{concl-compact00}.
	Since, on the other hand, for each $T>0$, 
	 \eqref{claimboundH2} implies
	$$\sup_{\e\in(0,\e_0]} \|z_\e\|_{H^1(0,T;L^2(I))}<\infty,$$
	it follows by interpolation with \eqref{concl-compact00} that $\sup_{\e\in(0,\e_0]} \|z_\e\|_{H^{1-\nu}(0,T;H^{3\nu}(I))}<\infty$
	 for each $\nu\in(0,1)$,
	hence \eqref{concl-compact0}.
\end{proof}

\begin{proof}[Proof of Lemma~\ref{lem.compact1}]
	We have
	$$\begin{aligned}
		G_\e(t)
		\le& \e^{-1}\int_I\int_0^{t/\e} \left(\int_{t-\e a}^t |z_t(s,\sigma)| d\sigma\right)^2\rho(s,a) dads  +\e^{-1}\int_I\int_{t/\e}^\infty u^2(t,s,a)\rho(s,a) dads
		\equiv G^{(1)}_\e(t)+G^{(2)}_\e(t).
	\end{aligned}$$
	To estimate $ G^{(1)}_\e(t)$, we write
	\be{estimG1a}
	\begin{aligned}
		G^{(1)}_\e(t)
		&\le\int_I\int_0^{t/\e}\left(\int_{t-\e a}^t |z_t(s,\sigma)|^2  d\sigma\right)a\rho(s,a) dads \le\int_0^{t/\e}\int_{t-\e a}^t \|z_t(\sigma)\|_2^2 a\|\rho(\cdot,a)\|_\infty d\sigma da \\
		&\le\int_0^t\|z_t(\sigma)\|_2^2 \left(\int_{(t-\sigma)/\e}^{t/\e} a\|\rho(\cdot,a)\|_\infty da\right)d\sigma \\
		&\le\int_0^{(t-\sqrt{\e})_+}\|z_t(\sigma)\|_2^2 \left(\int_{1/\sqrt{\e}}^\infty a\|\rho(\cdot,a)\|_\infty da\right)d\sigma
		+C\int_{(t-\sqrt{\e})_+}^t\|z_t(\sigma)\|_2^2d\sigma,
	\end{aligned}
	\ee
	where we used the second part of \eqref{hyp10} in the last inequality.
	To handle the last integral, we compute
	$$
	\int_T^{T+1}\int_{(t-\sqrt{\e})_+}^t \hspace{-0.2cm}\|z_t(\sigma)\|_2^2d\sigma dt 
	\le\int_0^{T+1} \|z_t(\sigma)\|_2^2 \biggl(\int_\sigma^{\sigma+\sqrt{\e}} dt\biggr)d\sigma\\
	\le \sqrt{\e}\int_0^{T+1} \|z_t(\sigma)\|_2^2d\sigma.
	$$
	Going back to \eqref{estimG1a} and using 
	 \eqref{globalboundz0stepGlobal},
	we get
	\be{estimG2a}
	\int_T^{T+1}G^{(1)}_\e(t)dt\le 
	C\bigl(\eta(\e)+\sqrt{\e}\bigr) \int_0^{T+1} \|z_t(\sigma)\|_2^2d\sigma\le C\bigl(\eta(\e)+\sqrt{\e}\bigr)\bar R.
	\ee
	To estimate $G^{(2)}_\e(t)$, we use \eqref{claimbound} and the first part of \eqref{hyp10} to write
	$$
	\begin{aligned}
		G^{(2)}_\e(t)& \le  C\bar R\e^{-1} \int_I\int_{t/\e}^\infty \rho(s,a) dads
		\le C\bar R\e^{1/2}t^{-3/2}\int_I\int_{t/\e}^\infty a^{3/2}\rho(s,a) dads \le C\bar R\e^{1/2}t^{-3/2}.
	\end{aligned}
	$$
	Integrating the latter for $t\in(T,T+1)$ and combining with \eqref{estimG2a}, we get \eqref{res.compact1}.
\end{proof}

\begin{proof}[Proof of Proposition~\ref{prop-cvE}]
	By Lemma~\ref{lem.compact0}, for any sequence $\e_j\to 0^+$, there exists a subsequence converging in $L^2_{loc}([0,\infty);H^2(I))$
	to some $\tilde z$, and we moreover have $\tilde z\in H^1_{loc}([0,\infty);$ $L^2(I))$ owing to Lemma \ref{lem.eps.zero}(i).
	By the proof of Lemma \ref{lem.eps.zero}(ii), $\tilde z$ is a weak solution of \eqref{mainstrong1-0}, hence $\tilde z=z_0$ 
	by  Lemma~\ref{lem.unique-weak}. This shows \eqref{concl-prop-cvE1}.
	
	Next write 
	$$E_\e(t)=\hat E_\e(t)+\frac12 G_\e(t),\quad\hbox{ with } \hat E_\e(t) :=\frac12\int_I (z_\e''(t))^2 ds+ \int_I F(z_\e'(t)) ds.$$
	By \eqref{concl-prop-cvE1} and Lemma~\ref{lem.compact1}, for any $t_2>t_1>0$ we have
	$$\lim_{\e\to 0} \int_{t_1}^{t_2}E_\e(\sigma) d\sigma=\lim_{\e\to 0} \int_{t_1}^{t_2}\hat E_\e(\sigma) d\sigma+
	\frac12\lim_{\e\to 0} \int_{t_1}^{t_2}G_\e(\sigma) d\sigma =\int_{t_1}^{t_2}E_0(\sigma) d\sigma.$$
	By the time monotonicity of $E_0$ and $E_\e$, dividing by $t_2-t_1$, it follows that
	\be{liminfsup}
	\liminf_{\e\to 0} E_\e(t_1) \ge E_0(t_2),\quad \limsup_{\e\to 0} E_\e(t_2) \le E_0(t_1).
	\ee
	On the other hand,  by \eqref{regul-0}, we have
	$E_0\in C([0,\infty))$.
	For fixed $t>0$, taking $t_1=t$ (resp., $t_2=t$) and letting $t_2\to t^+$ (resp., $t_1\to t^-$) in the first (resp., second) inequality of 
	\eqref{liminfsup}, we obtain 
	$$\liminf_{\e\to 0} E_\e(t) \ge E_0(t)\ge \limsup_{\e\to 0} E_\e(t),$$
	hence \eqref{concl-prop-cvE2} (the convergence being uniform for bounded $t$, owing to the monotonicity of $t\mapsto E_\e(t)$
	and Dini's theorem).
	
	Finally, by \eqref{eq.dissipation0} and {\eqref{eq.dissipation}, we have} $E'_\e \in L^1(0,T)$ 
	and  $\sup_{\e\in(0,\e_0)} \|E'_\e\|_{L^1(0,T)}<\infty$. 
		For any sequence $\e_i\to 0$, some subsequence of $E'_{\e_i}$ thus converges weakly 
		in the sense of measures to some limit $\mu$. On the other hand,  we know from \eqref{concl-prop-cvE2} that
	$E_\e$ converges to $E_0$ in $\D'((0,T))$  as $\e\to 0$. By uniqueness of 
	limits $\mu = E_0'$ and \eqref{concl-prop-cvE3} follows.
\end{proof}

\begin{proof}[Proof of Theorem~\ref{thm-cvE}(ii)] 
	First note that, since $\mathcal{S}$ is finite up to additive constants, we have
	\be{defeta1}
	\eta_1=\min\Bigl\{\tilde E(W);\ W\in\mathcal{S},\ |W'|\not\equiv 1\Bigr\}>0,
	\quad\hbox{where } \tilde E(W)=\frac12 \int_I (W'')^2 ds+\int_I F(W')ds.
	\ee
	Also, since the imbedding $H^2(I)\subset C^1(\overline I)$ is compact, we have
	\be{defeta2}
	\eta_2=\inf\Bigl\{\tilde E(W);\ W\in H^2(I),\ W'(0)=0\Bigr\}>0.
	\ee
	Set $\eta_0=\min(\eta_1,\eta_2)$.

	Assume that $Z'_0\equiv 1$ (the case $Z'_0\equiv -1$ is similar).
	Since $\lim_{t\to\infty} E_0(t)=\tilde E(Z_0)=0$ and  $\lim_{t\to\infty} \|z'_0(t)-1\|_\infty=0$, there exists $t_0>0$ such that 
	$E_0(t_0)<\eta_0/2$ and $z'_0(0,t_0)>1/2$.
	By the convergence property \eqref{concl-prop-cvE2} of the energy and the fact that $z_\e \to z_0$ in $C([0,t_0];C^1([0,L]))$ (cf.~Lemma \ref{lem.eps.zero}), 
	there exists $\bar\e_0\in(0,\e_0)$ such that, for all $\e\in(0,\bar\e_0)$, $E_\e(t_0)\le\eta_0/2$
	and $z'_\e(0,t_0)\ge 1/2$.
	In particular,
	$$\tilde E(Z_\e)\le E_\e(Z_\e)\le E_\e(t_0)\le \eta_0/2<\eta_1,$$
	hence $Z'_\e\equiv\pm 1$.
	Moreover, for any $\e\in(0,\bar\e_0]$, the case $Z'_\e\equiv -1$ cannot occur, since otherwise, by the continuity of 
	$t\mapsto z_\e(t)$ in $C^1(\bar I)$ 
	(cf.~Proposition~\ref{thmregtime}),
	there would exist $t_1>t_0$ such that $z'_\e(0,t_1)=0$, hence
	$$\eta_2\le\tilde E(z_\e(t_1))\le E_\e(t_1)\le E_\e(t_0)\le \eta_0/2:$$
	a contradiction. We have thus shown that, for all $\e\in(0,\bar\e_0]$, $Z'_\e\equiv 1$. The proof is complete.
\end{proof}

\begin{proof}[Proof of Theorem~\ref{thmglob}{\rm (iv)}]
	The argument is completely similar to that in the proof of Theorem~\ref{thm-cvE}(ii),
	but easier (and without requiring  assumption \eqref{hyprhoa}), 
	since we can just rely on the continuity property in Theorem~\ref{thmglob}{\rm (ii)} instead of Proposition~\ref{prop-cvE}).
	We therefore skip the details.
\end{proof}

\appendix

\def\appendixname{\hskip -1mm Appendix}

\section{Proof of Propositions~\ref{prop-pbm0} and \ref{prop-pbm0bEn}}

\label{Secz0}
\def\appendixname{\hskip -1mm}
	
	Consider the linear inhomogeneous problem
	\be{auxilu1ff}
	b\partial_tu+u''''=f'\ \hbox{in $Q_T$,\ with $u''=u'''-f=0$ on $\partial I$, \ $u(0)=u_0$,}
	\ee
	where $T>0$, $Q_T=(0,T)\times I$, $f\in L^2(0,T;H^1)$ and $u_0\in L^2$.
	
	\begin{defn} \label{def-auxilf}
		(i) A solution of \eqref{auxilu1ff} on $(0,T)$ is a function $u$ such that
		\begin{equation}\label{eq.zz0.weak}
			\left\{ 
			\begin{aligned}
				&u\in H^1(0,T;L^2)\cap L^2(0,T;H^2), \quad u(0)=u_0,\\
				&\int_0^T \int_I \bigl(b(s) \dt u\psi  + u''\psi'' + f\psi'\bigr)ds dt =0, \quad  
				\forall\psi \in L^2(0,T;H^2(I)).	
			\end{aligned}
			\right.
		\end{equation}
		Observe that the initial condition in \eqref{eq.zz0.weak} makes sense owing to $H^1(0,T;L^2)\subset C([0,T];L^2)$.
		\vskip 2pt
		
		(ii) A solution $z_0$ of \eqref{mainstrong1-0} on $(0,T)$ is a solution of \eqref{auxilu1ff} with $f=F'(z_0')$
		and $u_0=\phi$
		(note that $z_0\in L^2(0,T;H^2)$ implies $F'(z_0')\in L^2(0,T;H^1)$).
	\end{defn}

	\begin{rem} \label{rem-auxilf}
		(i) The definition \eqref{auxilu1ff} of (weak) solution is actually equivalent to that of $u$ being a strong solution,~i.e.
		\begin{equation}\label{eq.zz0.strong}
			\left\{ 
			\begin{aligned}
				&u\in H^1(0,T;L^2)\cap L^2(0,T;H^4), \quad u(0)=u_0,\\
				&b\partial_tu+u''''-f'=0 \hbox{ in $L^2(I)$ for a.e.~$t\in(0,T)$,}\\
				& u''=u'''-f=0 \hbox{ on $\partial I$  for a.e.~$t\in(0,T)$.}
			\end{aligned}
			\right.
			\ee
			Indeed, \eqref{eq.zz0.weak} implies $u''''=f' -b\partial_tu$ in the distributional sense in $(0,T)\times I$,
			hence $u\in L^2(0,T;H^4)$, and the other two conditions in \eqref{eq.zz0.strong} are satisfied owing to
			the identity
			$$\int_0^T\int_I \bigl(u''''-f'\bigr)\psi =\int_0^T\int_I \bigl(u''\psi'' + f\psi'\bigr)+\int_0^T[(u'''-f)\psi-u''\psi']_0^L,
			\ \ \psi \in L^2(0,T;H^2).$$
			Also, the converse implication readily follows from this identity.
			
			\smallskip
			
			(ii) Equation \eqref{auxilu1ff} has a most one solution
			(just subtract the equations for $u_1$ and $u_2$ and take $\psi=u_1-u_2$).
		\end{rem}

		We shall use the following linear result for \eqref{auxilu1ff}.
		In what follows, for $1<p<\infty$, we denote
		$$X_{p,T}:=W^{1,p}(0,T;L^2)\cap L^p(0,T;H^4),\quad
		X_{p,T,loc}:=W^{1,p}_{loc}((0,T];L^2)\cap L^p_{loc}((0,T];H^4)$$
and $B:L^2(0,T;H^1)\to L^2(0,T;\R^2)$ is the trace operator (namely, $(Bf)(t) = \{f(t,0),f(t,L)\}$ for a.e. $t \in (0,T)$).
		\begin{lem}\label{lemregulpq}
			Let $T>0$, $p>2$, $u_0\in H^2$, $f\in L^p(0,T;H^1)$ and assume that $Bf\in W^{\theta,p}(0,T;\R^2)$ with $\theta>1/8$.
			Then there exists a strong solution $u\in C([0,T];H^2)\cap X_{2,T}\cap X_{p,T,loc}$ of \eqref{auxilu1ff}
			such that $u(0)=u_0$. Moreover, for each $\eta\in(0,T)$, we have the estimate
			\be{EpTloc}
			\begin{aligned}
				\|\partial_tu\|_{L^p(\eta,T;L^2)}&+\|u''''\|_{L^p(\eta,T;L^2)}\le C(\eta,T)\bigl(\|f'\|_{L^p(0,T;L^2)}+\|B f\|_{W^{\theta,p}(0,T;\R^2)}
				+\|u\|_{L^p(0,T;L^2)}\bigr).
			\end{aligned}	
			\ee
			
		\end{lem}

		\begin{proof}
			Set $a=1/b \in H^2(I)$.
			We approximate the initial data by a sequence of smooth functions $u_{0,n}\in C^\infty(\bar I)$ such that 
			$\lim_{n\to\infty}\|u_{0,n}-u_0\|_{H^2}=0$ and consider the problem
			\be{auxilu1n}
			\partial_tu_n+au_n''''=af',\quad \hbox{with $u_n''=u_n'''-f=0$ on $\partial I$ and $u_n(0)=u_{0,n}$.}
			\ee
			We claim that, by \cite[Theorem 2.3]{DHP}, \eqref{auxilu1n} admits a (unique) solution
			$u_n\in X_{p,T}$.
			Indeed we check the applicability of that theorem with $m=q=2$ 
			and $\kappa_j =\frac{7-2j}{8}$ for $j=2,3$.
			Since $f'\in L^p(0,T;L^2)$ and since no compatibility conditions for the initial and boundary data are required owing to 
			$\kappa_j\le 1/q$, the result applies provided the boundary trace $Bf$ satisfies
			$B f\in F^\kappa_{p,2}(0,T;\R^2)$ with $\kappa=\kappa_3=1/8$,
			where $F^\kappa_{p,q}$ denotes the Triebel-Lizorkin space.
			To this end it suffices to use the fact that 
			$$W^{\theta,p}(0,T;\R^2)=F^\theta_{p,p}(0,T;\R^2)\subset F^\kappa_{p,2}(0,T;\R^2),\quad \theta>\kappa>0,\ 1<p<\infty$$
			(see~\cite[Chapter~2]{Tr1} and, e.g., \cite[section~1]{Si} for the first part and \cite[Theorem 1.2]{MV} for the second part.
			These are stated there in the whole Euclidean space, but the case of a smooth domain -- here just the interval $(0,T)$ -- follows 			by a standard extension property; see \cite[Chapter~2]{Tr2}).

			Now, by interpolation, we have $u_n\in W^{1/2,p}(0,T;H^2)\subset C([0,T];H^2)$.
			We claim that 
			\be{auxilu1n2}
			\max_{t\in[0,T]}\|(u_n-u_k)(t)\|_{H^2}+\int_0^T\int_I |\partial_t (u_n-u_k)|^2dsdt \le C\|u_{0,n}-u_{0,k}\|_{H^2}.
			\ee
			Indeed, set $w=u_n-u_k$, which is the solution of $\partial_t w+aw''''=0$
			with $w''=w'''=0$ on $\partial I$ and $w(0)=u_{0,n}-u_{0,k}$. For $0<t<T$ and $0<h<\min(t,T-t)$, 
			applying \eqref{eq.zz0.weak} with $u(t)=w(t+h)+w(t)$, $f=0$ and $\psi=h^{-1}(w(\tau+h)-w(\tau))$, we obtain
			$$\begin{aligned}
				h^{-1}\int_t^{t+h}& \int_I (w'')^2-h^{-1}\int_0^h \int_I (w'')^2
				=h^{-1}\int_0^t\int_I (w''(\tau+h))^2-(w''(\tau))^2\\
				&=\int_0^t\int_I (w(\tau+h)+w(\tau))''\frac{(w(\tau+h)-w(\tau))''}{h}=-\int_0^t\int_I b(w(\tau+h)+w(\tau))_t\frac{w(\tau+h)-w(\tau)}{h}.
			\end{aligned}$$
			Letting $h\to 0$ and using $u_n\in C([0,T];H^2)\cap H^1(0,T;L^2)$, we obtain
			$$\int_I (w'')^2(t)+\int_0^t\int_I b |\partial_tw(t)|^2=\int_I (w'')^2(0).$$
			Since also
			$\int_I w^2(t)\le 2\int_I w^2(0)+2T\int_0^T\int_I |\partial_tw|^2$,
			we deduce \eqref{auxilu1n2}.
			
			It follows from \eqref{auxilu1n2} that $(u_n)$ is a Cauchy sequence, hence converges, in $C([0,T];$ $H^2)$ $\cap H^1(0,T;L^2)$.
			Passing to the limit in the integral identity \eqref{eq.zz0.weak} for $u_n$, we obtain a solution $u$ of \eqref{auxilu1ff} in that 				class.
			Now let $\varphi\in C^1([0,T])$, with $\varphi=0$ on $[0,\eta/2]$ and $\varphi=1$ on $[\eta,T]$.
			The function $v:=u\varphi$ satisfies 
			$\partial_tv+av''''=g=af'\varphi+u\partial_t\varphi$ with $v''=0$, $v'''=\tilde f\equiv f\varphi$ on $\partial I$ and $v(0)=0$.
			Since $u\in C([0,T];H^2)\subset L^\infty(Q_T)$, hence $g\in L^p(0,T;L^2)$,
			and $B \tilde f\in W^{\theta,p}(0,T;\R^2)$,
			it follows from \cite[Theorem 2.3]{DHP} and uniqueness of solutions (cf.~Remark~\ref{rem-auxilf}(ii))
			that $v\in X_{p,T}$, hence $u\in X_{p,T,loc}$, with
			$$\|\partial_tv\|_{L^p(0,T;L^2)}+\|v''''\|_{L^p(0,T;L^2)}\le C(T)\bigl(\|g\|_{L^p(0,T;L^2)}+\|B \tilde f\|_{W^{\theta,p}(0,T;\R^2)}\bigr),$$
			which yields \eqref{EpTloc}.
		\end{proof}
			
	We now turn to the proof of Proposition~\ref{prop-pbm0}(i).
	Although one might directly apply a fixed point argument on problem \eqref{mainstrong1-0},
	it will be convenient to prove existence by taking advantage of the convergence result 
	in Lemma \ref{lem.eps.zero} for problem \eqref{mainweak},
	applied to a suitably defined kernel $\rho$ and past data $z_p$.
	We stress that we did {\it not} use Proposition~\ref{prop-pbm0} in the proof of Lemma \ref{lem.eps.zero}, so there is no circular reasoning.

	\begin{proof}[Proof of Proposition~\ref{prop-pbm0}(i)]
		We define
		\be{defrhob}
		\rho(s,a)=e^{-a}b(s) \quad\hbox{ and }\quad z_p(s,t)=\phi(s) \ \hbox{ for all $t\in(-\infty,0]$.}
		\ee
		Then $\rho, z_p$ satisfy all the assumptions of Lemma \ref{lem.eps.zero}
		and $\rhostar(s)=\int_0^\infty ae^{-a}b(s)da=b(s)$.
		It follows that there exists a global solution 
		\be{mainstrong1-0reg}
		z_0  \in C([0,T];C^1(\overline I))\cap L^\infty(0,\infty;H^2(I)),
		\ee
		with
		\be{mainstrong1-0reg2}
		z_0\in H^1(0,T;L^2(I))\cap L^2(0,T;H^4(I)),\quad 0<T<\infty,
		\ee
		of 
		$$
		\left\{\begin{aligned}
			b(s)\partial_t z_0 + z_0'''' &= \left(F'(z_0')\right)' ,&\hbox{ in $Q_\infty$}, \\
			z_0'''=0,\ z_0'''&=F'(z_0'),&s\in\partial I,\ t>0,\\
			z_0(s,0)&=\phi(s),&s\in I, 
		\end{aligned}
		\right.
		$$
		obtained as limit as $\e\to 0$ of solutions $z_\e$ of problems \eqref{mainweak}.
		
		\smallskip
		
		The uniqueness part of Proposition~\ref{prop-pbm0}(i) is a consequence of Lemma~\ref{lem.unique-weak}.
		
		\smallskip
		We next prove the additional regularity properties of the solution.
		Let $T>0$, fix any $p\in(2,\infty)$ and set $f:=F'(z_0')$.
		By \eqref{mainstrong1-0reg}, interpolation and Sobolev embedding, we have $z_0\in H^{1-\nu}(0,T;H^{4\nu})$
		for all $\nu\in(0,1)$, hence $f\in L^\infty(0,\infty;H^1)$, as well as $Bf\in H^k(0,T;\R^2)$ for all $k\in(0,5/8)$,
		with $\sup_{t\ge 1} \|Bf\|_{H^k(t,t+1;\R^2)}<\infty$.
		Therefore, by Sobolev embedding, there exists $\theta_p>1/8$ such that $Bf\in W^{\theta_p,p}(0,T;\R^2)$,
		with moreover
		\be{regul-01aa}
		\sup_{t\ge 1} \ \|Bf\|_{W^{\theta_p,p}(t,t+1;\R^2)}<\infty.
		\ee
		By Lemma~\ref{lemregulpq}, there exists a solution $u\in C([0,T];H^2)\cap X_{2,T}\cap X_{p,T,loc}$ of 
		\eqref{auxilu1ff}. By uniqueness of solutions (cf.~Remark~\ref{rem-auxilf}(ii)), we deduce that $u=z_0$. 
		Moreover, owing to estimate \eqref{EpTloc} in Lemma~\ref{lemregulpq} and the fact that
		$z_0\in L^\infty(0,\infty;H^2(I))$, we deduce \eqref{regul-01}.
	\end{proof}
	
	We next give the proof of the energy identity for problem \eqref{mainstrong1-0}.
	Although, at a formal level, the identity would readily follow by multiplying by $\partial_t z_0$ and integrating by parts,
	more care is needed in view of the available regularity of $z_0$.

	\begin{proof}[Proof of Proposition~\ref{prop-pbm0bEn}]
		(i) Let $t_2>t_1\ge 0$ and $0<h<t_2-t_1$. We write
		$$\begin{aligned}
			&2h^{-1}\int_{t_2}^{t_2+h}E_0(t)dt-2h^{-1}\int_{t_1}^{t_1+h}E_0(t)dt=h^{-1}\int_{t_2}^{t_2+h}\int_I ((z_0'')^2+2F(z_0'))-h^{-1}\int_{t_1}^{t_1+h}\int_I((z_0'')^2+2F(z_0'))\\
			&=h^{-1}\int_{t_1}^{t_2} \int_I \bigl[(z_0''(\tau+h))^2+2F(z_0'(\tau+h))\bigr]-
			\bigl[(z_0''(\tau))^2+2F(z_0'(\tau))\bigr]\\
			&=\int_{t_1}^{t_2} \int_I(z_0(\tau+h)+z_0(\tau))''\frac{(z_0(\tau+h)-z_0(\tau))''}{h}
			+2\frac{F(z_0'(\tau+h))-F(z_0'(\tau))}{h}.
		\end{aligned}$$
		On the other hand, by the weak formulation of \eqref{mainstrong1-0} (cf.~Definition~\ref{def-auxilf}), for all $\psi \in L^\infty(t_1,t_2;H^2(I))$, we have
		$$
		\begin{aligned}
			\int_{t_1}^{t_2} \int_I&\left(z_0(\tau+h)+z_0(\tau)\right)''\psi''
			+\bigl(F'(z_0'(\tau+h))+F'(z_0'(\tau))\bigr)\psi' + b\left(\partial_t z_0(\tau+h)+\partial_t z_0(\tau)\right)\psi =0.
		\end{aligned}
		$$
		Taking $\psi=h^{-1}(z_0(\tau+h)-z_0(\tau))$, we get
		$$\begin{aligned}
			&2h^{-1}\int_{t_2}^{t_2+h}E_0(t)dt-2h^{-1}\int_{t_1}^{t_1+h}E_0(t)dt=-\int_{t_1}^{t_2} \int_Ib\bigl(\partial_t z_0(\tau+h)+\partial_t z_0(\tau)\bigr) \frac{z_0(\tau+h)-z_0(\tau)}{h}\\
			& 
			+\int_{t_1}^{t_2}\int\Bigl[2\frac{F(z_0'(\tau+h))-F(z_0'(\tau))}{z_0'(\tau+h)-z_0'(\tau)}-F'(z_0'(\tau+h))-F'(z_0'(\tau))\Bigr]\frac{z_0'(\tau+h)-z_0'(\tau)}{h}\equiv J_1+J_2.
		\end{aligned}$$
		Since $z_0\in H^1(0,T;L^2)$, we have $\lim_{h\to 0}J_1=-2\int_{t_1}^{t_2} \int_Ib(\partial_tz_0)^2dsdt$.
		On the other hand, letting $M:= \sup_{Q_T}|z'_0|<\infty$ and using that
		$$\Bigl|\frac{F(X)-F(Y)}{X-Y}-F'(X)\Bigr|\le c|X-Y|,\quad |X|, |Y|\le M,$$
		for some constant $c=c(M,F)>0$, we get
		$$|J_2|\le 2c\int_{t_1}^{t_2}\int_I\frac{|z'_0(\tau+h)-z''_0(\tau)|^2}{h}\le 2cT\,\biggl(\,\sup_{\tau\in(t_1,T)} h^{-1/2}\|z'_0(\tau+h)-z'_0(\tau)\|_{L^2}\biggr)^2.$$
		Now, by \eqref{mainstrong1-0reg2} and interpolation, we have $z\in W^{3/4,p}_{loc}(0,T;H^1)$ for all $p\in(2,\infty)$, hence 
		$z'\in C^\alpha_{loc}(0,T;L^2)$ for all $\alpha\in(0,3/4)$, so that $\lim_{h\to 0}J_2=0$.
		Using that $z_0\in C([0,T];H^2)$, hence $E_0\in C([0,T])$, the assertion follows.
		
		\smallskip
		
		(ii) This follows immediately by integrating the equation in space.
	\end{proof}

	We shall finally prove Proposition~\ref{prop-pbm0}(ii).
	The $\omega$-limit set $\omega(z_0)$ is defined~by \eqref{defomegaze} with $\e=0$ and we set
	$$\tilde{\mathcal{S}}_K:=\{Z\in \mathcal{S};\ (Z,b)=K\},\quad K\in\R,$$
	where, as before, $\mathcal{S}$ is the set of steady states, i.e.~solutions of \eqref{eqstatz}.
	We use the following properties of $\omega(z_0)$.

	\begin{prop}\label{prop-pbm0b}
		Let the assumptions of Proposition~\ref{prop-pbm0} be in force.
		
		\vskip 1mm
		(i) The set $\omega(z_0)$ is a nonempty compact connected subset of $H^2(I)$.
		
		\vskip 1mm
		
		(ii) We have $\omega(z_0)\subset \tilde{\mathcal{S}}_{K_0}$, where $K_0:=(\phi,b)$.
	\end{prop}

	\begin{proof}
		\smallskip
		(i) This follows from the same argument as Lemma~\ref{prop-quasistat}(i).
		
		\smallskip
		
		(ii) By Proposition~\ref{prop-pbm0bEn}(i), we have
		\be{intz0cv}
		\int_0^\infty\int_I b(s) |\partial_tz_0|^2 dsdt\le E_0(0)<\infty.
		\ee
		Let $Z\in \omega(z_0)$. There exists a sequence $t_n\to\infty$ such that $z_0(t_n)\to Z$ in $H^2(I)$.
		Set $z_n(s,t)=z_0(s,t_n+t)$ and $Q=(0,1)\times I$
		and identify $Z(s,t)=Z(s)$.
		By \eqref{intz0cv} we have 
		$$\sup_{t\in(0,1)} \|z_n(\cdot,t)-z_n(\cdot,0)\|_{L^2(I))}^2 \le \int\int_Q  |\partial_tz_n|^2 dsdt\le C\int_{t_n}^\infty\int_I b(s) |\partial_tz_0|^2 dsdt\to 0$$
		as $n\to\infty$.
		Consequently,
		$\lim_{n\to\infty}\|z_n-Z\|_{L^\infty(0,1;L^2(I))}=0$.
		Since, on the other hand, \eqref{concl-compact0} guarantees that the sequence $z_n$ is 
		compact in $L^\infty(0,1;H^2(I))$, we deduce that $\lim_{n\to\infty}\|z_n-Z\|_{L^\infty(0,1;H^2(I))}=0$.
		We may thus pass to the limit in the weak formulation of \eqref{mainstrong1-0} to obtain,
		for any $v\in H^2(I)$, 
		$$
		\begin{aligned}
			\Bigl|\int_I  \bigl(Z'' v''+ F'(Z') v'\bigr)ds\Bigr| 
			&=\Bigl|\int\int_Q  \bigl(Z'' v''+ F'(Z') v'\bigr)dsdt\Bigr| 
			=\lim_n\Bigl|\int\int_Q  \bigl(z''_n v''+ F'(z_n') v'\bigr)dsdt\Bigr| \\
			&= \lim_n\Bigl|\int\int_Q  b\partial_tz_n vdsdt\Bigr| 
			\le \lim_n\left(\int\int_Q b(s) |\partial_tz_n|^2 dsdt\right)^{1/2}\|v\|_{L^2(I)}=0.
		\end{aligned}
		$$
		Consequently, $Z\in\mathcal{S}$, hence $Z\in\mathcal{S}_{K_0}$ in view of Proposition~\ref{prop-pbm0bEn}(ii).
	\end{proof}
	
	\begin{proof}[Proof of Proposition~\ref{prop-pbm0}(ii)]
		Based on Proposition~\ref{prop-pbm0b}, 
		this follows from the same argument as Theorem~\ref{thmglob}(iii). \end{proof}
	
	\begin{rem}\label{rem-pbm0stab}
		(i) Let us justify the statement in Remark~\ref{rem-mainth}(i) about the stability of the steady states $W'=\pm 1$.
		If $\phi\in H^2$ and $\|\phi''\|_2+\|\phi'-1\|_\infty$ is sufficiently small
		(the case with $\phi'+1$ is similar), then $\phi'(0)>0$ and $\tilde E(\phi)<\eta_0:=\min(\eta_1,\eta_2)$,
		where $\tilde E, \eta_1, \eta_2$ are defined in \eqref{defeta1}, \eqref{defeta2}. 
		On the other hand, by Proposition~\ref{prop-pbm0}(ii), the solution $z_0(t)$ of \eqref{mainstrong1-0}
		converges in $H^2$ to a steady state~$Z_0$.
		Therefore, $\tilde E(Z_0)=\lim_{t\to\infty} E_0(t)\le \tilde E(\phi)<\eta_1$, 
		hence $Z_0'\equiv 1$ or $-1$. Assume for contradiction that $Z'_0=-1$.
		Then $\lim_{t\to\infty} z'(0,t)=-1$ and, since $z_0\in C([0,\infty);C^1(\bar I))$,
		there would exist $t_0>0$ such that $z'_0(0,t_0)=0$, hence
		$\eta_2\le E_0(t_0)\le  \tilde E(\phi)<\eta_1$: a contradiction. 
		\smallskip
		
		(ii) As mentioned in Remark~\ref{rem-mainth}(i), if $W$ is a steady state such that $W'\not\equiv\pm 1$,
		then $W$ is unstable. More precisely there exist initial data $\phi$ arbitrarily close to $W$ in $H^2$
		such that $z_0$ converges to a steady-state $Z_0$ such that $Z_0'\not\equiv W'$.
		
		Indeed, since $\tilde E(W)>0$, we may choose a sequence $\phi_n\in H^2$ 
		such that $\lim_{n\to\infty}$ $\|\phi_n-W\|_{H^2}=0$ and $0<\tilde E(\phi_n)<\tilde E(W)$.
		Let $z_{0,n}$ be the corresponding solution of \eqref{mainstrong1-0} and $E_{0,n}(t)$ the corresponding energy function.
		By Proposition~\ref{prop-pbm0}(ii), $z_{0,n}(t)$ converges in $H^2$ to a steady state $Z_{0,n}$.
		Therefore, $\tilde E(Z_{0,n})=\lim_{t\to\infty} E_{0,n}(t)\le \tilde E(\phi_n)<\tilde E(W)$,
		hence $Z_{0,n}'\not\equiv W'$.
	\end{rem}

\noindent {\bf Declarations of interest:} none.

\bibliographystyle{abbrvnat}

\bibliography{biblio}
\end{document}